\documentclass[english,12pt]{article}

\pdfoutput=1

\usepackage[T1]{fontenc}
\usepackage[latin9]{inputenc}
\usepackage{amssymb}
\usepackage{babel}
\usepackage{amsthm}
\usepackage{verbatim}
\usepackage{amsrefs}
\usepackage{amsmath}
\usepackage{listings}
\usepackage{fullpage}
\usepackage{graphicx}
\usepackage{qtree}
\usepackage[usenames,dvipsnames]{color}
\usepackage{array}
\usepackage{enumerate}
\usepackage{polynom}
\usepackage{tikz}
\usepackage{fancyhdr}
\usepackage{etoolbox}
\usepackage{algorithm}
\usepackage{algorithmic}

\newtheorem{defin}{Definition}
\newtheorem{prop}{Proposition}
\newtheorem{theorem}{Theorem}
\newtheorem{lemma}{Lemma}
\newtheorem{cor}{Corollary}
\newtheorem{conj}{Conjecture}

\newcommand{\p}[1]{\left(#1\right)}
\newcommand{\st}[1]{\left\{#1\right\}}
\newcommand{\ab}[1]{\left|#1\right|}
\newcommand{\bk}[1]{\left[#1\right]}

\newcommand{\fl}[1]{\left\lfloor#1\right\rfloor}
\newcommand{\ceil}[1]{\left\lceil#1\right\rceil}

\newcommand{\quot}[1]{``#1''}

\newcommand{\limf}[3]{\lim_{#1\rightarrow#2}{#3}}

\newcommand{\limi}[2]{\limf{#1}{\infty}{#2}}

\newcommand{\eps}{\varepsilon}

\newcommand{\lb}{\hspace*{\fill}}

\DeclareMathOperator{\modd}{mod}

\DeclareMathOperator{\mex}{mex}

\newcommand{\fwd}{\p{\Longrightarrow}}
\newcommand{\bwd}{\p{\Longleftarrow}}
\renewcommand{\iff}{\Leftrightarrow}
\newcommand{\iffpf}[2]{
\begin{description}
\item[$\fwd$] #1
\item[$\bwd$] #2
\end{description}
}

\newcommand{\namehead}[3]{
\lstset{breaklines=true, morecomment=[l]{//}, frame=single, showstringspaces=false, numbers=left}
\begin{flushright}
Nathan Fox\\
#2\\
#3\\
\end{flushright}
\ifstrequal{#1}{.}{}{
\begin{center}
{\Large Homework #1}
\end{center}}
}

\usepackage{placeins}

\newcommand{\seq}{\p}
\newcommand{\INVARIANT}{\STATE\hspace{-0.155in}\textbf{Invariant:} }

\DeclareMathOperator{\per}{per}
\DeclareMathOperator{\pref}{pref}
\DeclareMathOperator{\SG}{SG}
\DeclareMathOperator{\ind}{ind}

\begin{document}
%
%
\title{On Aperiodic Subtraction Games with Bounded Nim Sequence}
\author{Nathan Fox\footnote{Department of Mathematics, Rutgers University, Piscataway, New Jersey,
\texttt{fox@math.rutgers.edu}
}}
\date{}

\maketitle

\begin{abstract}
Subtraction games are a class of impartial combinatorial games whose positions correspond to nonnegative integers and whose moves correspond to subtracting one of a fixed set of numbers from the current position.  Though they are easy to define, subtraction games have proven difficult to analyze.  In particular, few general results about their Sprague-Grundy values are known.  In this paper, we construct an example of a subtraction game whose sequence of Sprague-Grundy values is ternary and aperiodic, and we develop a theory that might lead to a generalization of our construction.
\end{abstract}

\section{Introduction}
We begin with the following definitions:
\begin{defin}\cite{bcg}
An \emph{impartial combinatorial game} is a game played by two players satisfying the following conditions:
\begin{itemize}
\item Both players have perfect information about the position
\item The players take turns making moves from the current position
\item The available moves do not depend on which player is to move
\item The first player to be unable to move loses
\item The game must end with one player losing in a finite number of moves
\end{itemize}
\end{defin}
We will also need the following definition:
\begin{defin}\cite{bcg}
The \emph{Sprague-Grundy function} of a position in an impartial combinatorial game is defined to be the minimal excluded nonnegative integer ($\mex$) of the Sprague-Grundy function values of all positions reachable from the current position.  (In particular, the Sprague-Grundy value of a position with no legal moves is zero.)
\end{defin}
The Sprague-Grundy function gives important information about positions in an impartial combinatorial games.  In particular, the player to move can force a win if and only if the Sprague-Grundy function of the current position is nonzero.  Such positions are called $N$-positions ($N$ for \quot{Next,} because the next player wins); positions with zero Sprague-Grundy function are called $P$-positions ($P$ for \quot{Previous,} because the previous player wins)~\cite{bcg}.

We now introduce subtraction games.
\begin{defin}\cite{bcg}
Let $S$ be a set (finite or infinite) of positive integers.  The \emph{subtraction game with subtraction set $S$} is the 
impartial combinatorial game with the following additional properties:
\begin{itemize}
\item The positions are enumerated by the nonnegative integers
\item In position $i$, a move is legal if and only if it is to position $i-s$ for some $s\in S$ with $s\leq i$
\end{itemize}
\end{defin}
Notice that a subtraction game is, in fact, an impartial combinatorial game, since the index of the position strictly decreases as the game progresses, and it is bounded below by $0$.  We will typically think of position $i$ in a subtraction game as a pile of $i$ counters, where a legal move is to take some element $s\in S$ counters from that pile, leaving $i-s$ counters for the other player.

We will focus mainly on the Sprague-Grundy values of subtraction games.  Since the subtraction set uniquely determines the subtraction game and vice versa, the Sprague-Grundy function can just as easily be thought of as a property of the subtraction set $S$.  Hence, we will use the notation $\SG_S$ to refer to the Sprague-Grundy function of the subtraction game with subtraction set $S$.  If the set $S$ is unambiguous, we will frequently drop the subscript and just write $\SG$.  The focus of this paper is on the sequence of Sprague-Grundy values, which is captured by the following definition.
\begin{defin}
Let $G$ be a subtraction game with subtraction set $S$.  The \emph{Sprague-Grundy sequence}, or \emph{Nim sequence}, of $G$ is the sequence $\seq{a_n}_{n\geq0}$ where $a_n=\SG_S\p{n}$.
\end{defin}
Like the Sprague-Grundy function, the Nim sequence can just as easily be thought of as a property of the subtraction set $S$.  Throughout this paper, we will use the terminology \quot{Nim sequence of a game} and \quot{Nim sequence of a set} interchangeably.

The best-known subtraction game is one-pile Nim.  In this game, $S$ is the set of natural numbers, and the Nim sequence is $a_n=n$ for all $n$.  Another common example of a subtraction game is colloquially known as \quot{twenty-one}.  In this game $S=\st{1,2,3}$ (and the typical starting position is with $21$ counters).  The Nim squence of $S$ in this case is $a_n=n\modd 4$ for all $n$.

In this paper, we pursue a subtraction game whose Nim sequence is bounded and not eventually periodic.  Clearly, if the Nim sequence is not bounded (as in Nim), then it is not eventually periodic.  But, many bounded Nim sequences are known to be eventually periodic (as in twenty-one).  In 2011, Fraenkel published a paper entitled \quot{Aperiodic Subtraction Games}, in which he proves that a certain impartial combinatorial game has an aperiodic, bounded Nim sequence.  His definition of subtraction game is different from ours, though, in that he also allows moves that reduce the value of the position by a multiplicative factor~\cite{fraenkel}.  To our knowledge, nobody has previously published an example of a subtraction game with aperiodic, bounded Nim sequence using our definition of subtraction game.

In Section~\ref{s:prelim}, we introduce some notation and concepts that we will use in our pursuit of an aperiodic, bounded Nim sequence for a subtraction game.  In particular, we will make use of some definitions and results on words (\ref{ss:words}), Beatty sequences (\ref{ss:beatty}), and digital representations (\ref{ss:digrep}).  In Section~\ref{s:known}, we introduce some results that are already known about Nim sequences of subtraction games, and we frame these results in the context of our problem.  In Section~\ref{s:traceable}, we define and explore the properties of the key tool that will allow us to solve the problem in a general context: representation words.  Then, in Section~\ref{s:apbd} we use representation words to find an example of a subtraction set with a bounded, aperiodic Nim sequence.  Finally, in Section~\ref{s:open}, we state some problems that remain open.

A Maple package implementing all of the algorithms in this paper, as well as some other tidbits, can be found at \texttt{http://math.rutgers.edu/$\sim$nhf12/subgames.txt}.

\section{Preliminaries}\label{s:prelim}
\subsection{Notation}\label{ss:notation}
In this paper, we will use some fairly standard notation for certain operations on sets of numbers.  First of all, for a set $A$, we will use $\ab{A}$ to denote the cardinality of $A$. Also, we have the following definition.
\begin{defin}
Let $A$ and $B$ be sets of numbers, let $n$ be a number, and let $f$ be a function that takes numerical arguments.  Then $A+B=\st{a+b\mid a\in A, b\in B}$, $A-B=\st{a-b\mid a\in A, b\in B}$, $A+n=\st{a+n\mid a\in A}$, $A-n=\st{a-n\mid a\in A}$, $n-A=\st{n-a\mid a\in A}$, $nA=\st{na\mid a\in A}$, and $f\p{A}=\st{f\p{a}\mid a\in A}$.
\end{defin}
We will also use the standard notation $\bk{n}$ to refer to the set $\st{1,2,\ldots,n}$.
\subsection{Combinatorics on Words}\label{ss:words}
Let $A$ be a set.  A \emph{word} over alphabet $A$ is a sequence of elements of $A$.  We allow words to be either finite or right-infinite (extending forever to the right but not the left).  We will use $\eps$ to denote the empty word.  For an alphabet $A$, $A^*$ denotes the set of finite words over $A$, and $A^\omega$ denotes the set of infinite words over $A$.  For a word $w$ and nonnegative integers $a<b$, $w\bk{a}$ denotes the $a^{th}$ symbol in $w$ (indexing from $0$), and $w\left[a..b\right)$ denotes the sequence of symbols in $w$ starting at position $a$ (inclusive) and ending at position $b$ (exclusive).

For words $u$ and $v$, we write $uv$ for the result of concatenating $u$ with $v$.  For a nonnegative integer $n$, $u^n$ is the result of concatenating $u$ with itself $n$ times (where $u^0=\eps$).  Also, $u^\omega$ denotes the result of concatenating $u$ with itself infinitely many times.  We will use product notation ($\prod$) to denote concatenation of an indexed family of words.

We will frequently think of Nim sequences as words, and we will use sequence and word notation interchangeably.  We will let $\Sigma$ denote the (infinite) alphabet of nonnegative integers.  Also, for nonnegative integers $k$, we will let $\Sigma_k$ denote the set $\st{0,1,\ldots,k}$.

We will occasionally wish to consider maps between sets of words known as \emph{morphisms}.  Given alphabets $A$ and $B$, a morphism is a map $\varphi:A^*\to B^*$ such that for all $u,v\in A^*$, $\varphi\p{uv}=\varphi\p{u}\varphi\p{v}$.  In particular, $\varphi\p{\eps}=\eps$, and for nonempty words $w$, $\varphi\p{w}$ is completely determined by $\varphi\p{a}$ for all $a\in A$.  If $A=B$, then we can iterate a morphism $\varphi:A^*\to A^*$.  Given $w\in A^*$, $\varphi^m\p{w}$ denotes the result of iterating $\varphi$ on input $w$ $m$ times, and
\[
\varphi^\omega\p{w}=\limi{m}\varphi^m\p{w}.
\]
The word $\varphi^m\p{w}$, if it exists, is known as the \emph{fixed point of $\varphi$ at $w$}.  Since the fixed point will typically be an infinite word, we will usually only consider fixed points at words of length $1$, i.e. at symbols of $A$.
\subsection{Beatty Sequences and Sturmian Words}\label{ss:beatty}
In order to prove that our main construction actually works, we will need to use some properties of Beatty sequences.
\begin{defin}\cite{rs}
Let $\alpha\in\mathbb{R}$ be irrational.  Define a sequence $b_n=\fl{n\alpha}$.  The sequence $\st{b_n}$ is called the \emph{Beatty sequence generated by $\alpha$}.
\end{defin}
Henceforth, $\alpha$ will refer to an arbitrary irrational number.  Here are some known facts about Beatty sequences.
\begin{prop}\label{prop:beattycomp}\cite{rs}
Let $n$ be a positive integer, and let $\alpha>1$.  Either $n$ is in the Beatty sequence generated by $\alpha$ or $n$ is in the Beatty sequence generated by $\p{1-\frac{1}{\alpha}}^{-1}$. 
\end{prop}
We have the following:
\begin{prop}\label{prop:beattyforbdiff}
Let $m$ and $n$ be integers (positive, negative, or zero) in the Beatty sequence generated by $\alpha$.  Then there exists an integer $\ell$ such that $m-n=\fl{\ell\alpha}$ or $m-n=\ceil{\ell\alpha}$.
\end{prop}
\begin{proof}
Since $m$ and $n$ are in the Beatty sequence generated by $\alpha$, there exist integers $a$ and $b$ such that $m=\fl{a\alpha}$ and $n=\fl{b\alpha}$.  So, there exist $r,s\in\p{0,1}$ such that $m=a\alpha-r$ and $n=b\alpha-s$.  So,
\[
m-n=\p{a\alpha-r}-\p{b\alpha-s}=\p{a-b}\alpha-\p{r-s}.
\]
Let $\ell=a-b$.  Notice that if $r>s$, then $m-n=\fl{\ell\alpha}$, since $r-s$ equals the fractional part of $\ell\alpha$.  If $r<s$, then $m-n=\ceil{\ell\alpha}$, since $r-s$ equals minus the fractional part of $1-\ell$.  These are all the possible cases, as required.
\end{proof}

The specific Beatty sequences we will work with are those generated by $\phi$ and $\phi^2$, where $\phi=\frac{1+\sqrt{5}}{2}$ is the golden ratio.  Since $\phi^2=\p{1-\frac{1}{\phi}}^{-1}$, it follows from Proposition~\ref{prop:beattycomp} that these two sequences partition the positive integers.  We have the following definition.
\begin{defin}\label{def:wyt}\cite{rs}, \cite{oeis201}, \cite{oeis1950}
The Beatty sequence generated by $\phi$ is known as the \emph{Lower Wythoff Sequence} (A000201 in OEIS).  The Beatty sequence generated by $\phi^2$ is known as the \emph{Upper Wythoff Sequence} (A001950 in OEIS).  We will denote the set of Lower Wythoff Numbers (starting from $1$) by $W_L$, and we will denote the set of Upper Wythoff Numbers (starting from $0$) by $W_U$.
\end{defin}
By our previous assertion, $W_L$ and $W_U$ are disjoint, and their union is all of the nonnegative integers.  (Note that our inclusion of $0$ in $W_U$ is nonstandard.  We do this because our construction later will treat $0$ like a member of the Upper Wythoff Numbers.)

While on the topic of Beatty sequences, we will define what it means for a word to be Sturmian.  There are many equivalent definitions for this; we choose the one most closely tied to Beatty sequences.
\begin{defin}\label{def:sturm}\cite{bal}
A binary word $w$ is Sturmian if there exist an irrational number $\alpha\in\p{0,1}$ and a real number $\beta\in\left[0,1\right)$ such that $w\bk{n}=\fl{\p{n+1}\alpha+\beta}-\fl{n\alpha+\beta}$ for all $n\geq0$ or $w\bk{n}=\ceil{\p{n+1}\alpha+\beta}-\ceil{n\alpha+\beta}$ for all $n\geq0$.  A general word $w$ over a two-letter alphabet $A$ is Sturmian if there exists a bijection $\varphi:\st{a,b}\to\st{0,1}$ such that $\varphi\p{w}$ is Sturmian.
\end{defin}
This definition gives rise to the following fact.
\begin{prop}\label{prop:beattysturm}\cite{cass}
The sequence of consecutive differences of the Beatty sequence generated by $\alpha$ is a Sturmian word on alphabet $\st{\fl{\alpha},\ceil{\alpha}}$.
\end{prop}

We conclude this subsection with a brief discussion of Sturmian morphisms.
\begin{defin}\label{def:sturmmorph}
A morphism $\varphi:\st{0,1}^*\to\st{0,1}^*$ is Sturmian if $\varphi^\omega\p{0}$ exists and is Sturmian.
\end{defin}
There are a number of characterizations of Sturmian morphisms~\cite{BeSe},~\cite{rich}.  The one we will use later is the following:
\begin{prop}\label{prop:sturmmorph}\cite{rich}
A morphism $\varphi:\st{0,1}^*\to\st{0,1}^*$ is Sturmian if and only if it has an infinite fixed point at $0$ and it can be written as a composition of the morphism $\varphi_1$, $\varphi_2$, $\varphi_3$, and $\varphi_4$ defined as follows:
\begin{itemize}
\item$\varphi_1\p{0}=0$ and $\varphi_1\p{1}=01$
\item$\varphi_2\p{0}=01$ and $\varphi_2\p{1}=0$
\item$\varphi_3\p{0}=0$ and $\varphi_3\p{1}=10$
\item$\varphi_4\p{0}=01$ and $\varphi_4\p{1}=1$
\end{itemize}
\end{prop}
\subsection{Digital Representations}\label{ss:digrep}
Digital representations, a well-studied topic~\cite{drmota}, will serve as a useful tool in our continued search for an aperiodic, bounded Nim sequence.  We have the following definition.
\begin{defin}\cite{drmota}
Let $\seq{a_i}_{i\geq0}$ be a strictly increasing sequence of positive integers with $a_0=1$.  Given a nonegative integer $n$ with $a_j\leq n<a_{j+1}$, for all $i<j$ choose a nonnegative integer $d_i$ such that
\[
n=\sum_{i=0}^j d_ia_i.
\]
If the sum
\[
\sum_{i=0}^j d_i
\]
is taken as small as possible, then we call the sequence $d_jd_{j-1}\cdots d_0$ the \emph{$\seq{a_i}$-expansion or $\seq{a_i}$-representation of $n$}.
\end{defin}
This representation is unique.  Given $n$ with $a_j\leq n<a_{j+1}$,
the $\seq{a_i}$-expansion of $n$ can be obtained greedily: The first digit will be $\fl{\frac{n}{a_j}}$, then the remaining digits will be the $\seq{a_i}$-expansion of $\p{n\modd a_j}$~\cite{drmota}.

The most common example of a digital representation is the base $b$ representation for integers $b\geq2$.  In this case, $a_i=b^i$.  Another well-known example is the Zeckendorf representation, where $a_i=F\p{i+2}$ (where $F\p{m}$ denotes the $m^{th}$ Fibonacci number).  For example, the Zeckendorf representation of $19$ is $101001$, since $19=13+5+1$.

Two important properties of the Zeckendorf representation of any $n$ are that all digits are $0$ or $1$ and that no two ones can be consecutive.  The converse is also true: any bit string satisfying these two conditions is a valid Zeckendorf representation~\cite{zeck}.

We will find these next definitions useful when dealing with digital representations.
\begin{defin}
Let $\seq{a_i}_{i\geq0}$ be a strictly increasing sequence of positive integers with $a_0=1$.  We call any such sequence a \emph{representing sequence}.
\end{defin}
\begin{defin}
Let $\seq{a_i}_{i\geq0}$ be a representing sequence, and let $n$ be a positive integer.  Let $j$ be such that $a_j\leq n<a_{j+1}$.  We say that $j$ is the $\emph{index}$ of $n$ with respect to $\seq{a_i}$.  If the sequence is unambiguous, we will use the notation $\ind\p{n}$ for the index of $n$.
\end{defin}
Notice that $\ind\p{n}$ is one less than the number of digits in the $\seq{a_i}$-representation of $n$.  So, the notion of index generalizes logarithms to arbitrary representations, since $\ind\p{n}=\fl{\log_b\p{n}}$ in the base $b$ representation.

We will now derive some useful results about digital representations.
\begin{prop}\label{prop:cutoff}
Let $\seq{a_i}$ be a representing sequence, and let $n$ be a positive integer.  Let $j=\ind\p{n}$.  Let $d_jd_{j-1}\cdots d_0$ be the $\seq{a_i}$-representation of $n$.  Then, $\p{d_j-1}d_{j-1}d_{j-2}\cdots d_0$ is the $\seq{a_i}$-representation of $n-a_j$.
\end{prop}
\begin{proof}
Recall the greedy algorithm for constructing digital representations.  This algorithm tells us that $d_j=\fl{\frac{n}{a_j}}$.  Letting $e_j$ denote the digit in the $a_j$ position in the $\seq{a_i}$-representation of $n-a_j$, we have similarly that $e_j=\fl{\frac{n-a_j}{a_j}}=d_j-1$, as required.  The greedy algorithm also tells us that $d_{j-1}d_{j-2}\cdots d_0$ is the $\seq{a_i}$-representation of $n\modd a_j$.  Since $\p{n-a_j}\modd a_j=n\modd a_j$ we see that these last $j$ digits will be the same in the representation of $n-a_j$, as required.
\end{proof}
\begin{prop}\label{prop:lexdig}
Let $\seq{a_i}$ be a representing sequence, and let $m$ and $n$ be positive integers.  If $m<n$, then the $\seq{a_i}$-representation of $m$ precedes that of $n$ lexicographically (when padding $m$ with zeroes so the representations are the same length).
\end{prop}
\begin{proof}
Assume that $m<n$.  Let $j=\ind\p{m}$.  If $n\geq a_{j+1}$, then the representation of $n$ has a nonzero digit in a higher position than $m$, so the representation of $m$ precedes that of $n$ lexicographically.
So, we can assume without loss of generality that $a_j\leq m<n<a_{j+1}$.  Assume for a contradiction that the representation of $n$ precedes that of $m$ lexicographically. Furthermore, assume that $j$ is the minimal index where such a pair $\p{m,n}$ exists.  Let $c$ be the positive integer such that $ca_j\leq m<\p{c+1}a_j$.  This $c$ is the first digit in the representation of $m$.  The first digit in $n$ must be at least $c$, so it must equal $c$ by our order assumption.  So, we see that $m-ca_j$ and $n-ca_j$ are both less than $a_j$, and  the representation of $n-ca_j$ precedes that of $m-ca_j$ lexicographically, since those representations are the same as the representations of $m$ and $n$ with the first digit removed.  This is a contradiction, since $m-ca_j<n-ca_j$.  Therefore, the $\seq{a_i}$-representation of $m$ precedes that of $n$ lexicographically, as required.
\end{proof}
\begin{prop}\label{prop:maxdig}
Let $\seq{a_i}$ be a representing sequence.  For each fixed nonnegative integer $j$, let $d_j$ be the first digit in the $\seq{a_i}$-representation of $a_{j+1}-1$.  Then, for all nonnegative integers $n$, the digit in the $a_j$ place in the $\seq{a_i}$-representation of $n$ is at most $d_j$.
\end{prop}
\begin{proof}
For each nonnegative integer $n$, we will say that $n$ has a \emph{violation} in position $j$ if the digit in the $a_j$ place in the $\seq{a_i}$-representation of $n$ exceeds $d_j$.  We will show, by induction on $n$, that there are no violations.  It is clear that $n=0$ has no violations.  Now, assume that no integer less than $n$ has a violation in any position.  Let $\ell=\ind\p{n}$.  As a consequence of Proposition~\ref{prop:lexdig}, there is not a violation in the $a_\ell$ position of $n$, since $n\leq a_{\ell+1}-1$, which has $d_\ell$ in that position.  Let $d$ be the digit in the $a_\ell$ position of $n$.  Then, the representation of $n-da_\ell$ is the same as that of $n$ but with the first digit removed.  We assumed that $n-da_\ell$ has no violations.  Therefore, $n$ has no violations, as required.
\end{proof}

Before we prove our key results about digital representations, we need the following definition.
\begin{defin}
Let $\seq{a_i}$ be a representing sequence, and let $n$ be a nonnegative integer.  For positive integer $m$, we say that $n$ is \emph{$m$-volatile} with repsect to $\seq{a_i}$ if the $\seq{a_i}$-representation of $n+1$ ends in (at least) $m$ zeroes.  (If $m=1$, we just say that $n$ is \emph{volatile}.)  If the $\seq{a_i}$-representation of $n+1$ does not end in $0$, we say that $n$ is \emph{non-volatile}.
\end{defin}
Volatility is a generalization of needing to carry when adding $1$.
We have the following facts about volatility.
\begin{prop}\label{prop:vprez}
Let $\seq{a_i}$ be a representing sequence, and let $n$ be a nonnegative integer.  Let $j=\ind\p{n}$, and let $m$ be a positive integer.  Then,
\begin{enumerate}
\item If $n-a_j$ is $m$-volatile, then $n$ is $m$-volatile.
\item If $n<a_{j+1}-1$ and $n$ is $m$-volatile, then $n-a_j$ is $m$-volatile.
\end{enumerate}
\end{prop}
\begin{proof}\lb
\begin{enumerate}
\item Assume that $n-a_j$ is $m$-volatile.  Then, the $\seq{a_i}$-representation of $n-a_j+1$ ends in $m$ zeroes.  If $n+1<a_{j+1}$, then Proposition~\ref{prop:cutoff} says that the last $j$ digits of the representation of $n+1$ will be the same as those of $n-a_j+1$.  In particular, this ends in $m$ zeroes, which means that $n$ is $m$-volatile, as required.  If $n+1\geq a_{j+1}$, though, then $n+1=a_{j+1}$.  This means that its representation is a single $1$ followed by at least $m$ zeroes.  Thismeans that $n$ is $m$-volatile, as required.
\item We will prove the contrapositive.  Assume that $n-a_j$ is not $m$-volatile.  Then, the $\seq{a_i}$-representation of $n-a_j+1$ does not end in $m$ zeroes.  We are given that $n+1<a_{j+1}$, so Proposition~\ref{prop:cutoff} says that the last $j$ digits of the representation of $n+1$ will be the same as those of $n-a_j+1$.  This does not end in $m$ zeroes, which means that $n$ is not $m$-volatile, as required.
\end{enumerate}
\end{proof}
\begin{prop}\label{prop:zv2}
Let $\seq{a_i}$ be a representing sequence, and let $n$ be a nonnegative integer.  If $n$ is volatile and the $\seq{a_i}$-representation of $n$ does not end in $a_1-1$, then $n$ is $2$-volatile.
\end{prop}
\begin{proof}
We will prove the contrapositive.  Let $n$ be volatile but not $2$-volatile, and consider $n+1-a_1$.  We see that the $\seq{a_i}$-representation of this number is the same as that of $n$, except that the last digit in its representation is zero.  So, if $d$ is the last digit in the representation of $n$, we have that $n+1-a_1=n-d$.  So, $d=a_1-1$, as required.
\end{proof}
Here is another definition we will use frequently, often in conjunction with volatility.
\begin{defin}
Let $\seq{a_i}$ be a representing sequence, and let $n$ be a nonnegative integer.  If the $\seq{a_i}$-representation of $n$ ends in $0$, we say that $n$ is a \emph{zend} (for \textbf{z}ero \textbf{end}).
\end{defin}
As a corollary to Proposition~\ref{prop:zv2}, every volatile zend is $2$-volatile.

\section{Known Results about Subtraction Games}\label{s:known}
The following proposition is the first key observation in our search for an aperiodic bounded Nim sequence.  The proof is a routine pigeonhole argument; it has been relegated to Appendix~\ref{app:nsrp}.
\begin{prop}\label{prop:sinf}
Let $S$ be a finite set of positive integers.  There exist $u,v\in\Sigma^*$ such that the Nim sequence of $S$ equals $uv^\omega$.
\end{prop}
We can take $u$ and $v$ in Proposition~\ref{prop:sinf} to have minimal length.  When we do this, we will call $u$ the \emph{prefix} for $S$ and $v$ the \emph{period} for $S$.  We will denote the period by $\per\p{S}$ and the prefix by $\pref\p{S}$.  We have an algorithm for finding $\per\p{S}$ and $\pref\p{S}$ whose details and proof of correctness can be found in Appendix~\ref{app:alg}.

The following proposition will allow us to assume without loss of generality that the greatest common divisor of the elements in our subtraction sets is $1$.  Again, the proof is routine and can be found in Appendix~\ref{app:nsrp}.
\begin{prop}\label{prop:gcd}
Let $S$ be a subtraction set, and let $g\geq2$ be an integer. For all $n\geq0$, we have $\SG_{gS}\p{n}=SG_S\p{\fl{\frac{n}{g}}}$.
\end{prop}

Here is another result that we will use frequently in our analysis.  We will refer to this result as \emph{Ferguson's Theorem}.
\begin{theorem}[Ferguson]\cite{bcg}\label{thm:ferguson}
Let $G$ be a subtraction game on subtraction set $S$.  Let $s=\min S$.  Let $SG$ denote the Sprague-Grundy function on $G$.  Then, for all $i$, $SG\p{i}=0$ if and only if $SG\p{i+s}=1$.
\end{theorem}
We will actually prove the following theorem, which generalizes (and implies) Ferguson's Theorem.
\begin{theorem}\label{thm:genferguson}
Let $S$ be a set of positive integers, and let $s=\min S$.  Let $k$ equal the smallest positive multiple of $s$ not in $S$.  For all integers $i$ with $0\leq i<k-1$, if $SG\p{m}=i$, then $SG\p{m+s}=i+1$, and for all integers $i$ with $0<i\leq k-1$, if $SG\p{m}=i$, then $SG\p{m-s}=i-1$.
\end{theorem}
\begin{proof}
Consider the first value from $0$ to $k-1$ in the Nim sequence of $S$ that violates the claim.  Let it be in position $m$.  There are three cases to consider.
\begin{description}
\item[Followed by wrong value:] Let $0\leq i<k-1$, let $SG\p{m}=i$, and assume for a contradiction that $SG\p{m+s}\neq i+1$.  Since $SG\p{m-sh}=i-sh$ for all $h$ from $0$ to $i$, we have that $SG\p{m+s}>i$.  So, there exists $t\in S$ such that $SG\p{m+s-t}=i+1$ (or else $SG\p{m+s}=i+1$).  But then $SG\p{m-t}=i$, a contradiction.
\item[Preceded by too small value:] Let $0<i\leq k-1$, let $SG\p{m}=i$, and assume for a contradiction that $SG\p{m-s}=i-h$ for $h\geq2$.  There exists $t\in S$ such that $SG\p{m-t}=i-h+1$ (or else $SG\p{m}<i$).  But then $SG\p{m-s-t}=i-h$, a contradiction.
\item[Preceded by too large value:] Let $0<i\leq k-1$, let $SG\p{m}=i$, and assume for a contradiction that $SG\p{m-s}>i-1$.  There exists $t\in S$ such that $SG\p{m-s-t}=i-1$ (or else $SG\p{m-s}\leq i-1$).  But then $SG\p{m-t}=i$, a contradiction.
\end{description}
\end{proof}
In the future, we will refer to Theorem~\ref{thm:genferguson} as \emph{Generalized Ferguson's Theorem}.  We will also use the following definition in the future:
\begin{defin}\label{def:ferg}
We say a word $w$ on alphabet $\Sigma_k$ ($k\geq2$) is \emph{Fergusonian} if for all $0\leq i<k-1$, every $i$ is followed by an $i+1$ and every $i+1$ is preceded by an $i$.  Furthermore, we will say that $w$ is \emph{strongly Fergusonian} if it contains no consecutive $k$'s.
\end{defin}
This is a natural definition, since, by Theorem~\ref{thm:genferguson}, any Nim sequence that takes values $\st{0,1,\ldots, k}$ and whose subtraction set contains $\st{1,\ldots,k-1}$ must be Fergusonian.  In fact, any such Nim sequence must be strongly Fergusonian, since it is legal to move from position $i+1$ to position $i$ for all $i$ (meaning there are no consecutive $k$'s in the Nim sequence).  

Ferguson's Theorem has an important corollary, which implies that searching for an aperiodic, binary Nim sequence will be in vain.  The proof can be found in Appendix~\ref{app:nsrp}.
\begin{cor}\label{cor:ternary}
Let $S$ be a subtraction set with bounded, aperiodic Nim sequence.  Then, the Sprague-Grundy function on $S$ takes at least three values.  (So, the Nim sequence is not binary.)
\end{cor}
\section{Representation Words}\label{s:traceable}
In this section, we introduce representation words, which form a crucial link between Nim sequences and digital representations.  The combinatorial properties of representation words will allow us to make much progress toward a general form for some aperiodic, bounded Nim sequences.  We have the following definition.
\begin{defin}
Let $\seq{a_i}$ be a representing sequence.  For each nonnegative integer $n$, let $d_n$ denote the last digit in the $\seq{a_i}$-representation of $n$.  The \emph{representation word} for $\seq{a_i}$ is the word $w$ over $\Sigma_{a_1}$ obtained as follows:
\[
w\bk{n}=\begin{cases}
a_1 & n\text{ $2$-volatile}\\
d_n & \text{otherwise}.
\end{cases}
\]
Given a respresentation word $w$, we call $\seq{a_i}$ the \emph{representation sequence} of $w$.
\end{defin}
Note that, in general, a representation word $w$ may have multiple representation sequences.  For example, the representing sequences $\p{1,2,5,13}$ and $\p{1,2,5,10,13}$ both have representation word $\p{0101201012012}^\omega$.

It may seem a bit unnatural to force the word to take value $a_1$ in position $n$ when $n$ is $2$-volatile.  Both the value $a_1$ and the condition of $n$ being $2$-volatile seem a bit aribitrary.  We will soon see that this modification is the key that allows us to connect these words to Nim sequences.

Here are a few properties of representation words.  These properties will be immediate consequences of corresponding properties of digital representations.
\begin{prop}\label{prop:trivrep}
Let $\seq{a_i}$ be a representing sequence, and let $w$ be the representation word of $\seq{a_i}$.  Fix $h\in\Sigma_{a_1}$.  Then,
\begin{itemize}
\item If $h=0$, then every $h$ in $w$ (other than the initial $0$ in $w$) is preceded by an $a_1$ or an $a_1-1$.
\item If $h<a_1-1$, then every $h$ in $w$ is followed by an $h+1$ or an $a_1$.
\item If $0<h<a_1$, then every $h$ in $w$ is preceded by an $h-1$.
\item If $h=a_1$, then every $h$ in $w$ is followed by an $a_1$ or a $0$.
\end{itemize}
\end{prop}
\begin{proof}
Fix a nonnegative integer $n$, and let $h=w\bk{n}$.  We will consider each case separately.
\begin{itemize}
\item Let $h=0$, and let $n>0$.  Then, $n-1$ is volatile.  If $n-1$ is not $2$-volatile, then the $\seq{a_i}$-representation of $n-1$ ends in $a_1-1$ by Proposition~\ref{prop:zv2}, so $w\bk{n-1}=a_1-1$, which is permissible.  If $n-1$ is $2$-volatile, then $w\bk{n-1}=a_1$, which is also permissible.
\item Let $h<a_1-1$.  Then, $n$ is not $2$-volatile.  So, by Proposition~\ref{prop:zv2}, $n$ is non-volatile, since its $\seq{a_i}$-representation ends in $h$.  If $n+1$ is $2$-volatile, then $w\bk{n+1}=a_1$, which is permissible.  If $n+1$ is not $2$-volatile, then the $\seq{a_i}$-representation of $n+1$ ends in $h+1$, so $w\bk{n+1}=h+1$, which is also permissible.
\item Let $0<h<a_1$.  Then, the $\seq{a_i}$-representation of $n$ ends in $h$.  This implies that the $\seq{a_i}$-representation of $n-1$ ends in $h-1$, so $w\bk{n-1}=h-1$, as required.
\item Let $h=a_1$.  Then, $n$ is $2$-volatile.  If $n+1$ is $2$-volatile, then $w\bk{n+1}=a_1$, which is permissible.  If $n+1$ is not $2$-volatile, then the $\seq{a_i}$-representation of $n+1$ ends in $0$, so $w\bk{n+1}=0$, which is also permissible.
\end{itemize}
\end{proof}

The key combinatorial property of representation words is given by the following theorem.
\begin{theorem}\label{thm:repwdtrace}
Let $w$ be a word, and let $\seq{a_i}$ be a representing sequence.  The following are equivalent:
\begin{enumerate}
\item $w$ is the representation word of $\seq{a_i}$
\item\label{i:rep2} $w$ is infinite, and for each nonnegative integer $n$, we have
\[
w\bk{n}=\begin{cases}
n & n<a_1\\
a_1 & n=a_{j+1}-1\text{ and }j\geq1\\
w\bk{n-a_j} & \text{otherwise},
\end{cases}
\]
where $j=\ind\p{n}$.
\end{enumerate}
\end{theorem}
\begin{proof}\lb
\iffpf{
Let $w$ be the representation word of $\seq{a_i}$.  Fix $n\geq0$, and let $j=\ind\p{n}$.  There are three cases to consider.
\begin{description}
\item[$n<a_1$:] In this case, the $\seq{a_i}$-representation of $n$ has one digit, and that digit is $n$ itself.  Hence, $n$ is not $2$-volatile, so $w\bk{n}=n$, as required.
\item[$n=a_{j+1}-1\text{ and }j\geq1$:] The $\seq{a_i}$-representation of $a_{j+1}$ ends in two zeroes.  So, $n$ is $2$-volatile, so $w\bk{n}=a_1$, as required.
\item[Neither of previous cases:] By Proposition~\ref{prop:cutoff}, the $\seq{a_i}$-representations of $n$ and $n-a_j$ share the same last $j$ digits.  In particular, they share their last digit (since $j\geq1$).  Also, since $n\neq a_{j+1}-1$, Proposition~\ref{prop:vprez} implies that $n$ is $2$-volatile if and only if $n-a_j$ is $2$-volatile.  The combination of these facts implies that $w\bk{n}=w\bk{n-a_j}$, as required.
\end{description}
}
{
Let $v$ be the representation word of $\seq{a_i}$, and let $w$ be a word satisfying condition~\ref{i:rep2}.  By the forward direction $v$ also satisfies condition~\ref{i:rep2} for the same sequence $\seq{a_i}$.  Hence, $v=w$ (since $w$ is infinite), as required.
}
\end{proof}
Theorem~\ref{thm:repwdtrace} gives an efficient method to compute $w\bk{n}$, assuming that the sequence $\seq{a_i}$ grows fast enough and can be computed efficiently.  In the case that $w$ is a Nim sequence, this gives a way to efficiently compute the Sprague-Grundy function.
\subsection{Fergusonian Representation Words}
We have seen that every aperiodic Nim sequence over $\Sigma_k$ of a subtraction set $S$ containing $\bk{k-1}$ is strongly Fergusonian.  Since we will eventually show that certain representation words are, in fact, aperiodic Nim sequences of this type, we will narrow our attention onto strongly Fergusonian representation words.

Our first result will be a characterization of when a representation word is Fergusonian (not necessarily strongly Fergusonian).
\begin{prop}\label{prop:traceferg}
Let $\seq{a_i}$ be a representing sequence, and let $w$ be the reprsentation word of $\seq{a_i}$.  The following are equivalent:
\begin{enumerate}
\item\label{it:ferg1} $w$ is Fergusonian
\item\label{it:ferg2} Every $a_1$ in $w$ follows either an $a_1-1$ or an $a_1$.
\item\label{it:ferg3} $w\bk{n}=a_1$ if and only if $n$ is a volatile zend.
\end{enumerate}
\end{prop}
\begin{proof}\lb
\begin{description}
\item[$\p{\ref{it:ferg1}}\Rightarrow\p{\ref{it:ferg2}}$:] Let $w$ be Fergusonian.  Then, every $k$ must follow either a $k$ or a $k-1$ (since all other values are followed by a prescribed value), as required.
\item[$\p{\ref{it:ferg2}}\Rightarrow\p{\ref{it:ferg3}}$:] By Proposition~\ref{prop:zv2}, every volatile zend is $2$-volatile.  So, if $n$ is a volatile zend, then $w\bk{n}=a_1$ without any further assumptions on $w$.  Now, assume that every $a_1$ in $w$ follows either an $a_1-1$ or an $a_1$.  Let $w\bk{n}=a_1$ for some $n$.  We wish to show that $n$ is a zend.  We have that $w\bk{n-1}=a_1$ or $w\bk{n-1}=a_1-1$.  In the first case $n-1$ is $2$-volatile, and hence volatile.  In the second case, $n-1$ is also volatile.  So, since $n-1$ is volatile, $n$ is a zend, as required.
\item[$\p{\ref{it:ferg3}}\Rightarrow\p{\ref{it:ferg1}}$:] Assume that $w\bk{n}=a_1$ if and only if $n$ is a volatile zend.  We claim that $w$ is Fergusonian.  Let $0\leq h<a_1-1$.  We claim that $w\bk{n}=h$ if and only if $w\bk{n+1}=h+1$.  The reverse implication is part of Proposition~\ref{prop:trivrep}.  Proposition~\ref{prop:trivrep} also says that if $w\bk{n}=h$, then $w\bk{n+1}=\in\st{h+1,a_1}$.  If $w\bk{n+1}=a_1$, then $n+1$ is a zend.  This means that $n$ is volatile.  It cannot be that $n$ is $2$-volatile, as then we would have $w\bk{n}=a_1$.  But, Proposition~\ref{prop:zv2} tells us that we cannot have $n$ being volatile without $n$ being $2$-volatile.  Therefore, $w\bk{n+1}=h+1$, as required.
\end{description}
\end{proof}
Our next major result is a method of producing infinite strongly Fergusonian representation words.
\begin{theorem}\label{thm:fergtracegen}
Let $k\geq2$ be an integer, and let $w_1=01\cdots\p{k-1}$.  For all $i>1$, let
\[
w_i=\p{\prod_{j=1}^{i-1}w_{i-j}^{p_{i,i-j}}}k
\]
for some nonnegative integers $p_{i,\ell}$ with $1\leq\ell<i$.  If $p_{i,\ell}\geq p_{i+1,\ell}$ for all pairs $\p{i, \ell}$ with $1\leq\ell<i$ and if $p_{i,1}\geq1$ for all $i$,
then
\[
w=\limi{i}{w_i}
\]
is a strongly Fergusonian representation word, and its representation sequence is
\[
\seq{1,k=\ab{w_1},\ab{w_2},\ab{w_3},\ldots}.
\]
\end{theorem}
The proof of Theorem~\ref{thm:fergtracegen} is rather long and tedious, so it has been relegated to Appendix~\ref{app:fergtracegen}.
\subsection{Connection to Nim Sequences}
Now we will see the connection between strongly Fergusonian representation words and Nim sequences of subtraction games.
Throughout this subsection, let $\seq{a_i}$ be a representing sequence with $a_1=k$, and let $w$ be the representation word of $\seq{a_i}$.  Furthermore, let $\seq{a_i}$ be such that $w$ is strongly Fergusonian.  We will now define some sets based on $\seq{a_i}$.  In all cases, we will omit the sequence if there is no ambiguity.  Let
\[
T\p{\seq{a_i}}=\st{a_i-1\mid i\geq2}\cup\bk{k-1}.
\]
Also, let $W\p{\seq{a_i}}$ denote the set of volatile zends in the $\seq{a_i}$-representation, and let $N\p{\seq{a_i}}$ denote the set of non-volatile zends.  Let $V\p{\seq{a_i}}=W\cup\bk{k-1}$.  Note that $T\subseteq V$.  Write $L\p{\seq{a_i}}=\mathbb{N}\setminus N$.
We make the following observations:
\begin{prop}\label{prop:scontt}
Let $S$ be a set of positive integers containing $\bk{k-1}$.  If $w$ is the Nim sequence for $S$, then $\st{1}\cup\p{S+1}$ contains a representation sequence for $w$.
\end{prop}
\begin{proof}
Let $U=\st{1}\cup\p{S+1}\setminus\st{2,3,\ldots,k-1}$.  Notice that $\st{1,k}\subset U$.  Define a sequence $\seq{c_i}$ by $c_0=1$, $c_1=k$, and for $i\geq2$, let
\[
c_i=1+\min_{s\in S,\;s>c_{i-1}-1}\p{\begin{array}{l}\text{The Nim sequences of }S\cap\bk{s-1}\text{ and }S\cap \bk{s}\\\text{\hspace{0.2in}differ before or in position }s\end{array}}.
\]
Clearly, this sequence is contained in $U$.  We claim that $\seq{c_i}$ is a representation sequence for $w$.

What we will prove is that for all $n\geq1$, $\seq{c_i}_{0\leq i\leq n}$ is a representation sequence for $\p{w\left[0..c_n\right)}^\omega$ (as long as $c_n$ is defined, which it will be if $w$ is not purely periodic).  The proof will be by induction on $n$.  Notice that $\p{1,k}$ is in fact a representation sequence for $\p{w\left[0..k\right)}^\omega=\p{01\cdots\p{k-1}}^\omega$.  Now, assume that $\p{c_i}_{0\leq i\leq n-1}$ is a representation sequence for $\p{w\left[0..c_{n-1}\right)}^\omega$.  By our construction of $c_n$, it follows that
\[
w\bk{c_n-1}\neq\p{\p{w\left[0..c_{n-1}\right)}^\omega}\bk{c_n-1}
\]
(since stopping $S$ at or before $c_n-1$ must modify the Nim sequence before or in position $c_n-1$, and it cannot possibly modify the sequence before position $c_n-1$).  Since we are taking the mex of more things to produce $w\bk{c_n-1}$, it must be that we increased the value in this position.  Since $w$ is strongly Fergusonian, this means that $w\bk{c_n-1}=k$, as required for $\seq{c_i}_{0\leq i\leq n}$ to be a representation sequence of $\p{w\left[0..c_{n-1}\right)}^\omega$.  The other proprerties of representation sequences follow from the fact that we give $\p{w\left[0..c_n\right)}^\omega$ the period $w\left[0..c_n\right)$.
\end{proof}
\begin{prop}\label{prop:convabsmain}
Assume that $w$ is the Nim sequence for $T$.  Then, $\p{N+T}\cap N=\emptyset$.
\end{prop}
\begin{proof}
We know that, $w\bk{i}=0$ if and only if $i\in N$.  Assume for a contradiction that there exists $s\in T$ such that $i+s\in N$ for some $i\in N$.  Then, $w\bk{i+s}=0$.  But, there is a legal move from position $i+s$ to a position whose Sprague-Grundy value is $0$ (namely position $i$).  Hence, $w\bk{i+s}\neq0$, a contradiction.  Therefore, $\p{N+T}\cap N=\emptyset$, as required.
\end{proof}
We will now work towards our main result in this subsection, which will be a sort of converse to Proposition~\ref{prop:convabsmain}.
First, we will prove a few useful lemmas about our sets.
\begin{lemma}\label{lem:nnins}
We have $L\subseteq N+T$.
\end{lemma}
\begin{proof}
Let $n\in L$.  There are two cases to consider.
\begin{description}
\item[$n\notin W$:] In this case, the representation of $n$ does not end in $0$ (since $n\notin N$).  Let the last digit of this representation be $i$.  Since $0<i<k$, we have $i\in T$.  Hence, $n-i$ is a zend, and it is non-volatile, since $n-i+i$ has all the same digits except for the last one.  So, since $n=\p{n-i}+i$, we have $n\in N+T$, as required.
\item[$n\in W$:] In this case, the representation of $n+1$ ends in at least two zeroes by Proposition~\ref{prop:zv2}.  Let $j$ be the number of zeroes at the end of the representation of $n+1$.  Consider $n-\p{a_j-1}$.  Since $n-\p{a_j-1}=\p{n+1}-a_j$, we obtain a representation will all the same digits as that of $n+1$, but the $a_j$ digit of $n+1$ is decreased by $1$ (possibly to $0$).  Hence, $n-\p{a_j-1}-1$ is $2$-volatile, so $w\bk{n-\p{a_j-1}-1}=a_1$.  Since $w$ is strongly Fergusonian, $w\bk{n-\p{a_j-1}}=0$, so $n-\p{a_j-1}\in N$.  We also have $a_j-1\in T$.  This gives
\[
n=\p{n-\p{a_j-1}}+\p{a_j-1}\in N+T,
\]
as required.
\end{description}
\end{proof}
\begin{lemma}\label{lem:nsnn}
Let $S$ be a set satisfying $T\subseteq S$.  If $\p{N+S}\cap N=\emptyset$, then $N+S=L$.
\end{lemma}
\begin{proof}
We already know that $L\subseteq N+S$, since $L\subseteq N+T$ by Lemma~\ref{lem:nnins} and $T\subseteq S$.  Also, we have that $N+S\subseteq L$, since we are told that $\p{N+S}\cap N=\emptyset$.
\end{proof}
\begin{lemma}\label{lem:vsv0}
Let $S$ be a set of positive integers.
If $\p{N+S}\cap N=\emptyset$, then $\p{W+S}\cap W=\emptyset$.
\end{lemma}
\begin{proof}
We are given that
\[
\p{N+S}\cap N=\emptyset.
\]
Subtracting $1$ from both sets yields
\[
\p{\p{N-1}+S}\cap\p{N-1}=\emptyset.
\]
Since $w$ is strongly Fergusonian, we have that $W\subseteq N-1$.  So, we obtain
\[
\p{W+S}\cap W=\emptyset,
\]
as required.
\end{proof}
We can now prove our main theorem in this subsection.
\begin{theorem}\label{thm:absmain}
Let $I$ be a set
satisfying $T\subseteq I
$ and $\p{N+I}\cap N=\emptyset$.  Then,
\begin{itemize}
\item For all sets $S$ with $T\subseteq S\subseteq I$, $w$ is the Nim sequence for $S$.
\item For all $i$, if $a+a_i-S_i\subseteq L$ for all $a\in N$ less than $a_i$, then $\p{w\left[0..a_i\right)}^\omega$ is the Nim sequence for the set $S_i=\st{s\in S\mid s<a_i}$.
\end{itemize}
\end{theorem}
\begin{proof}
First, we will prove that the Nim sequence of $S$ equals $w$.  We know that the zeroes in $w$ are in precisely the positions in $N$.
We claim that the zeroes of the Sprague-Grundy function of $S$ are also in those positions.  To prove this claim, we will show that position $n$ is a $P$-position if and only if $n\in N$.  The proof will be by induction on $n$.  If $n=0$, then the player to move loses instantly, so that is a $P$-position.  Since $0\in N$, this is the required outcome.  Now, assume that the claim holds for all $m<n$.  There are two cases to consider.
\begin{description}
\item[$n\in N$:] We will show that $n$ is a $P$-position.  By Lemma~\ref{lem:nsnn}, $N+I=L$.  So, since $S\subseteq I$, for all $s\in S$ with $s\leq n$, $n-s\notin N$.  By the inductive hypothesis, all of the positions $n-s$ are $N$-positions.  In other words, all legal moves from position $n$ are to $N$-positions.  This implies that $n$ is a $P$-position, as required.
\item[$n\notin N$:] We will show that $n$ is an $N$-position.  Since $T\subseteq I$,  Lemma~\ref{lem:nsnn} implies that $N+T=L$.  So, since $T\subseteq S$, there exists $s\in S$ such that $n-s\in N$.  By the inductive hypothesis, $n-s$ is a $P$-position.  In other words, there is a legal move from $n$ to a $P$-position.  This implies that $n$ is an $N$-position, as required.
\end{description}
Therefore, the zeroes of the Sprague-Grundy function of $S$ are in the same positions as the zeroes in $w$.
Now, recall that $w$ is Fergusonian.  So, for all $1\leq a<k$, every $a$ in $w$ is preceded by an $a-1$, and every $a-1$ in $w$ is followed by an $a$.  The same property holds for the $a$'s of the Sprague-Grundy function of $S$ by Theorem~\ref{thm:genferguson}.

Since $w$ is strongly Fergusonian, every $k$ in $w$ follows a $k-1$ and is followed by a $0$.  So, all that remains to verify is that every value of the Sprague-Grundy function of $S$ lies in $\st{0,1,\ldots, k}$.  The only positions left to classify are the positions where $w$ has a $k$.  These positions are precisely the set $W$.
By Lemma~\ref{lem:vsv0}, $\p{W+I}\cap W=\emptyset$.
So, if $n\in W$, this, since $S\subseteq I$, we have that all legal moves from $n$ are to positions not in $W$.  Hence, all legal moves are to positions with Sprague-Grundy values less than $k$.  Therefore, the Sprague-Grundy function at $n$ is at most $k$, so, since it is not less than $k$, it must equal $k$, as required.

We have shown that that Nim sequence of $S$ equals $w$.  We now wish to show that the Nim sequence of $S_i=\p{w\left[0..a_i\right)}^\omega$ for all $i$, under the additional assumption that for all $a<a_i$ with $a\in N$, $a+a_i-S_i\subseteq L$.
This proof will be by induction on $n$, the current position.
If $n<a_i$, then the Nim sequence matches that of $S$ (since the only legal moves thus far lie in $S_i$), as required.  Now, assume that $n\geq a_i$ and that we have the given Nim sequence through position $n-1$.  Let the proposed Sprague-Grundy value at $n$ be $m$.  The length of our proposed period for $S_i$ is $a_i$.  So, the legal moves from position $n$ can fall into only the current period and the previous period.  Let $a=n\modd a_i$, and note that $m$ depends only on $a$.
There are three cases to consider.
\begin{description}
\item[$m=0$:] In this case, the Sprague-Grundy value at $a$ is $0$, so we must have $a\in N$.  We will now show that for all $s\in S_i$, the Sprague-Grundy value of $n-s$ is nonzero.  If $s\leq a$, the Sprague-Grundy function of $n-s$ is the same as the Sprague-Grundy function of $a-s$, which is nonzero because $a\in N$ and $s\in S$.  If $s>a$, then the Sprague-Grundy function of $n-s$ is the same as the Sprague-Grundy function of $a+a_i-s$.
But, we are given that $a+a_i-S_i\subseteq L$, so $a+a_i-s\in L$, so this Sprague-Grundy value is nonzero, as required.
\item[$1\leq m<k$:] This case follows from Theorem~\ref{thm:genferguson} and the previous case.
\item[$m=k$:] In this case, the Sprague-Grundy value at $a$ is $k$, so we must have $a\in W$.
By Theorem~\ref{thm:genferguson} and the previous cases, the Sprague-Grundy value cannot be less than $k$.
Hence, it will suffice to show that for all $s\in S_i$, the Sprague-Grundy value of $n-s$ is not equal to $k$.  If $s\leq a$, the Sprague-Grundy function of $n-s$ is the same as the Sprague-Grundy function of $a-s$, which is not $k$ because $a\in W$ and $s\in S$.  If $s>a$, then the Sprague-Grundy function of $n-s$ is the same as the Sprague-Grundy function of $a+a_i-s$.  But, since $a_i>a$, we have that $a+a_i\in W$ also by Lemma~\ref{prop:vprez}, so this is not $k$, as required.
\end{description}
Therefore, the Nim sequence of $S_i$ is $\p{w\left[0..a_i\right)}^\omega$, as required.
\end{proof}
When we apply Theorem~\ref{thm:absmain}, we will commonly have $I=V$, which will mean that we can take $S=T$, $S=V$, or anything in between in those cases.  We conclude this section by giving a sufficient condition for the property that $a+a_i-S_i\subseteq L$ for all $a<a_i$ with $a\in N$.
\begin{prop}\label{prop:2aain}
If $\p{N+S}\cap N=\emptyset$ and $a_{i+1}>2a_i$ (for some specific $i$), then for all $a<a_i$ with $a\in N$, $a+a_i\in N$.  In particular, $a+a_i-S_i\subseteq L$.
\end{prop}
\begin{proof}
Let $a_{i+1}>2a_i$, and let $a<a_i$ with $a\in N$.  We have $a+a_i<a_i+a_i=2a_i<a_{i+1}$.  Since both of these inequalities are strict and since all quantities involved are integers, we have $a+a_i<a_{i+1}-1$.  So, by Lemma~\ref{prop:vprez}, since $a$ is non-volatile, $a+a_i$ is non-volatlie.  Also, $a+a_i$ is a zend, since $a$ is a zend and $a_i$ has its only nonzero digit in a higher place than the highest nonzero place in $a$.  Therefore, $a+a_i\in N$, as required.

Now, assume for a contradiction that $a+a_i-S_i\not\subseteq L$.  Then, there exists $s\in S_i$ such that $a+a_i-s\in N$.  Let $n=a+a_i-s$.  We then have $n+s\in N$, which contradicts Lemma~\ref{lem:nsnn}.  Therefore $a+a_i-S_i\subseteq L$, as required.
\end{proof}
\section{Finding an Aperiodic, Bounded Nim Sequence}\label{s:apbd}
\subsection{First Attempt: The Greedy Algorithm}\label{ss:greedy}
We have built up a detailed machine for analyzing aperiodic, bounded Nim sequences of subtraction sets.  Unfortunately, we still have not shown that such subtraction sets exist.  An obvious first attempt to find one is to construct a set by repeatedly adding elements and observing that the Nim sequences of the partial sets behave in a desireable manner.  The basis of this approach will be Algorithm~\ref{alg:greedy}.
\begin{algorithm}
\caption{$EXTEND\_SET\p{S,start,k}$}
\label{alg:greedy}
\begin{algorithmic}
\REQUIRE $k\geq2$ integer
\REQUIRE $S$ finite set of positive integers with $\pref\p{S}=\eps$ and $\max\p{\per\p{S}}\leq k$
\REQUIRE $start>\max\p{S}$ positive integer
\FOR {$i$ \textbf{from} $start$ \TO $\infty$}
   \IF {$\pref\p{S\cup\st{i}}=\eps$ \AND $\per\p{S\cup\st{i}}\neq\per\p{S}$ \AND $\max\p{\per\p{S\cup\st{i}}}\leq k$}
	\RETURN $S\cup\st{i}$
   \ENDIF
\ENDFOR
\end{algorithmic}
\end{algorithm}
\FloatBarrier
This algorithm takes as input a set $S$ with a (purely) periodic Nim sequence bounded by $k$.  It returns a set $S'=S\cup\st{i}$ for some $i\geq start$ such that $S'$ has a (purely) periodic Nim sequence bounded by $k$.  If no such extension of $S$ exists, this algorithm runs forever.

Our main observation is the following.
\begin{prop}\label{prop:limalg}
Let $k\geq2$ be an integer.  Let $\seq{a_i}_{i\geq0}$ be a strictly increasing sequence of positive integers, and for all $i$, let $S_i=\st{a_j\mid j\leq i}$.  Let $S=\st{a_j\mid j\geq0}$.  Furthermore, assume that $a_{i+1}>\ab{\per\p{S_i}}$ for all $i$.  If Algorithm~\ref{alg:greedy} returns $S_{i+1}$ on input $\p{S_i, a_{i+1}, k}$ for all $i$, then the Nim sequence of $S$ equals
\[
\limi{i}{\per\p{S_i}}.
\]
(In particular, this limit exists.)
\end{prop}
\begin{proof}
We will show that $\per\p{S_i}$ is a strict prefix of $\per\p{S_{i+1}}$ for all $n$.  This will suffice to show that the limit exists, and from there, it will be clear that this limit equals the Nim sequence for $S$, since for each position, only some finite subset of $S$ (some $S_i$) will be the legal set of moves.  Consider the prefix of $\per\p{S_{i+1}}$ of length $\ab{\per\p{S_i}}$.  Since $a_{i+1}>\ab{\per\p{S_i}}$, the only legal moves in positions in that prefix correspond to subtractions in $S_i$.  Hence, this prefix exactly equals $\per\p{S_i}$, as required.
\end{proof}
All of our tools so far relate to representation words.  So, we give a modified algorithm, Algorithm~\ref{alg:greedyrep} that requires the resulting Nim sequence to be a representation word.
\begin{algorithm}
\caption{$EXTEND\_REP\_SEQ\p{\seq{a_i}, start}$}
\label{alg:greedyrep}
\begin{algorithmic}
\REQUIRE $\seq{a_i}$ finite representing sequence
\REQUIRE Representation word of $\seq{a_i}$ is the Nim sequence of $T\p{\seq{a_i}}$
\STATE Let $v$ be the representation word of $\seq{a_i}$
\STATE Let $n$ be the maximum index $i$ such that $a_i$ is defined.
\REQUIRE $start>a_n$ positive integer
\FOR {$j$ \textbf{from} $start$ \TO $\infty$}
  \STATE Let $\seq{b_i}=\p{a_0,a_1,\ldots,a_n,j}$
  \STATE Let $w$ be the representation word of $\seq{b_i}$
   \IF {$w\neq v$ and $w$ is the Nim sequence of $T\p{\seq{b_i}}$}
	\RETURN $\seq{b_i}$
   \ENDIF
\ENDFOR
\end{algorithmic}
\end{algorithm}
\FloatBarrier

We are going to attempt to use Algorithms~\ref{alg:greedy} and~\ref{alg:greedyrep} to find an aperiodic ternary Nim sequence.  As a consequence of Theorem~\ref{thm:genferguson}, we will not search for subtraction sets containing $\st{1,2}$.  Nor will we concentrate on $\st{1,3}$, since the Nim sequence of $\st{1,3}$ equals the Nim sequence of $\st{1}$.  From this point forward, the subtraction sets we consider will start with $\st{1,4}$.
%
We make the following observation.
\begin{prop}\label{prop:nodoub}
Let $S$ be the set of positive integers congruent to $1$ modulo $3$, and let $S_i$ denote the subset of $S$ consisting of the smallest $i+1$ elements of $S$.  The Nim sequence of $S_i$ is $\p{01\p{012}^i}^\omega$, and the Nim sequence of $S$ is $01\p{012}^\omega$.
\end{prop}
\begin{proof}
Let $\seq{a_i}$ be the representing sequence defined by
\[
a_i=\begin{cases}
1 & i=0\\
3i-1 & i>0.
\end{cases}
\]
Notice that $01\p{012}^\omega$ is the representation word of $\seq{a_i}$.  The non-volatile zends are the numbers congruent to $2$ mod $3$, along with $0$.  Denote this set by $N$.  We see that $\p{N+S}\cap N=\emptyset$.  Therefore, by Theorem~\ref{thm:absmain}, the Nim sequence of $S$ is $01\p{012}^\omega$.

To show that the Nim sequence of $S_i$ is $\p{01\p{012}^i}^\omega$, we must show that for all $a<a_i$ with $a\in N$ we have $a+a_i-S_{i-1}\nsubseteq N$.  Notice that if $a\in N$, then $a\not\equiv1\p{\modd3}$.  If $i=0$, then $S_{i-1}$ is empty, so $a+a_i-S_{i-1}\nsubseteq N$, as required.  If $i>0$, then we have for all $s\in S_{i-1}$,
\[
a+a_i-s\not\equiv1+2-1\p{\modd3}\equiv2\p{\modd3}.
\]
Hence, $a+a_i-s\notin N$, since $a+a_i-s\neq0$.  Therefore, the Nim sequence of $S_i$ is $\p{01\p{012}^i}^\omega$ by Theorem~\ref{thm:absmain}, as required.
\end{proof}
Notice that the set $S$ in Proposition~\ref{prop:nodoub} satisfies Proposition~\ref{prop:limalg}.  This presents an irritating conundrum.  We wanted to use Algorithm~\ref{alg:greedy} to find an aperiodic ternary Nim sequence.  Instead, while the periods never stabilized, we obtained something with an eventual period that we had never encountered before.  To combat this, we have the following proposition.
\begin{prop}\label{prop:aperdoub}
Let $k\geq2$ be an integer.  Let $\seq{a_i}_{i\geq0}$ be a strictly increasing sequence of positive integers, and for all $i$, let $S_i=\st{a_j\mid j\leq i}$.  Let $S=\st{a_j\mid j\geq0}$.  Furthermore, assume that $a_{i+1}>2\ab{\per\p{S_i}}$ for all $i$.  If these parameters satisfy Proposition~\ref{prop:limalg}, then the Nim sequence returned by Algorithm~\ref{alg:greedy} will be aperiodic.
\end{prop}
To prove this, we use the following lemma.
\begin{lemma}\label{lem:aperdoub}
Let $\st{v_i}_{i\geq1}$ be a sequence of words in $\Sigma^*$.  Let $w_0\in\Sigma^*$, and for all $i\geq1$, let $w_i=w_{i-1}^2v_i$.  Let
\[
w=\limi{i}{w_i}.
\]
(This limit exists because each $w_i$ is a prefix of the next.)  If $w$ is eventually periodic, then $w$ is purely periodic, and for sufficiently large $n$, $v_i$ is a power of the period of $w$.
\end{lemma}
\begin{proof}
Assume that $w$ is eventually periodic.  Then, $w=pu^\omega$ for some $p,u\in\Sigma^*$.  We can assume that we have taken $p$ and $u$ to be of minimal length.  For sufficiently large $i$, $w_i=pu^mu'$ for some $m\geq1$ and $u'$ a proper prefix of $u$.  Then,
\[
w_{i+1}=w_i^2v_{i+1}=\p{pu^mu'}^2v_{i+1}=pu^mu'pu^mu'v_{i+1}.
\]
We also have that $w_{i+1}=pu^{m'}u''$ for some $m'\geq1$ and $u''$ a proper prefix of $u$.  Hence, by minimality of $u$, we have that $u'p$ is a power of $u$.  So, by minimality of $p$, we have $p=\eps$, so $u'=\eps$ (as if $p\neq\eps$, then we could shift $u$ one index to the left while shortening $p$).  This implies that $w$ is purely periodic.  Furthermore, it implies that $w_i=u^m$ and $w_{i+1}=u^{m'}=u^{2m}v_{i+1}$.  Therefore, $v_{i+1}$ is a power of the period $u$. This holds for sufficiently large $i$, as required.
\end{proof}
We will now prove Proposition~\ref{prop:aperdoub}.
\begin{proof}
We will show that $\per\p{S_i}^2$ is a strict prefix of $\per\p{S_{i+1}}$ for all $i$.   This, along with the detalis of Algorithm~\ref{alg:greedy} and Lemma~\ref{lem:aperdoub}, will imply that the limit Nim sequence is aperiodic.  Consider the prefix of $\per\p{S_{i+1}}$ of length $2\ab{\per\p{S_i}}$.  Since $a_{i+1}>2\ab{\per\p{S_i}}$, the only legal moves in positions in that prefix correspond to subtractions in $S_i$.  Hence, this prefix exactly equals $\per\p{S_i}^2$, as required.
\end{proof}

Now, what happens if we try to extend $\st{1,4}$ greedily, using Algorithm~\ref{alg:greedy} as a depth first search, making sure we double the period every time?  The sequence we obtain is
\[
1, 4, 12, 28, 73, 163, 343, 867, 1915, 4011, 8203,
\]
and the next term is at most (and probably equal to) $16587$.  There is not yet a proof that this sequence can be extended forever without backtracking, but evidence indicates that it likely can be and that the resulting sequence approximately doubles with each successive term.

We notice that Algorithm~\ref{alg:greedyrep} also yields this subtraction set when seeded with input $\seq{1}$ (and when the next term is required to at least double each time).  Let $\seq{a_i}$ denote the representation sequence of the word obtained via this algorithm, which begins
\[
1, 2, 5, 13, 29, 74, 164, 344, 868, 1916, 4012, 8204,
\]
and probably $16588$.  We notice the following properties of this sequence:
\[
\begin{cases}
a_i=2a_{i-1} & i=1\\
a_i=2a_{i-1}+1 & i=2\\
a_i=2a_{i-1}+3 & i=3,4\\
a_i=2a_{i-1}+16 & i=5,6,7\\
a_i=2a_{i-1}+180 & i=8,9,10,11\text{ (and probably $12$)}\\
\cdots
\end{cases}
\]
As of the time of writing this paper, neither the subtraction set sequence nor the representing sequence we obtained was in OEIS.  Also, neither were the sequence $2,4,17,181$ (with or without proper multiplicity) nor the sequence $0,1,3,16,180$ (with or without multiplicity).  But, we notice that $1=0+1$, $3=1+2$, $16=3+13$, and $180=16+164$.  In all cases, the first term in the sum is the previous term in the sequence, and the second term is in our representing sequence.

Also, we notice that for all these terms for all $i\geq2$, either $a_i=3a_{i-1}-2a_{i-2}$ or $a_i=3a_{i-1}-a_{i-2}$.
In the next subsection, we will build a sequence based on these pattern and analyze its properties.
\subsection{Second Attempt: Ternary Word Construction}\label{ss:ternary}
We begin this section by defining a sequence of ternary words.  Let $w_1=01$, and for all $i\geq2$, let
\[
w_i=w_{i-1}^2\p{\prod_{j=1}^{i-2}w_{i-1-j}}2.
\]
Let
\[
w=\limi{i}{w_i}.
\]
Since each $w_i$ is a prefix of $w_{i+1}$, this limit is well-defined.  Note that $w$ is aperiodic by Lemma~\ref{lem:aperdoub}.  Let $\seq{a_i}$ be the representing sequence defined by $a_0=1$ and $a_i=\ab{w_i}$ for all $i\geq1$.  By Theorem~\ref{thm:fergtracegen}, $w$ is the representation word of $\seq{a_i}$, and $w$ is strongly Fergusonian.
%
%

We will now obtain a characterization of this sequence.  This characterization will involve Fibonacci numbers.  Let $F\p{n}$ denote the $n^{th}$ Fibonacci number.  We will use the following recurrences at various points in this subsection:
\begin{lemma}\label{lem:fibrec}
The following recurrences hold regarding Fibonacci numbers:
\begin{enumerate}
\item\label{rec:0} For all $n\geq2$,
\[
2F\p{n}=F\p{n+1}+F\p{n-2}.
\]
\item\label{rec:1} For all $n\geq1$,
\[
F\p{2n+1}=2F\p{2n-1}+F\p{2n-2}.
\]
\item\label{rec:2} For all $n\geq2$,
\[
F\p{2n+1}=3F\p{2n-1}-F\p{2n-3}.
\]
\item\label{rec:3} For all $n\geq0$,
\[
F\p{2n}=\sum_{i=0}^{n-1}F\p{2i+1}.
\]
\item\label{rec:4} For all $n\geq0$,
\[
F\p{2n+1}=1+\sum_{i=1}^{n}F\p{2i}.
\]
\item\label{rec:5} For all $n\geq1$,
\[
F\p{2n+1}=F\p{2n-1}+\sum_{i=0}^{n-1}F\p{2i+1}.
\]
\end{enumerate}
\end{lemma}
The proofs of all of these lemmas are routine applications of the fundamental Fibonacci recurrence ($F\p{n}=F\p{n-1}+F\p{n-2}$), so we omit these proofs.  We now give the characterization of $\seq{a_i}$.
\begin{prop}\label{prop:wlen}
For all $i\geq1$, $\ab{w_i}=a_i=F\p{2i+1}$.
\end{prop}
\begin{proof}
Notice that $a_1=2=F\p{3}$.  Now, for $i\geq2$ we have
\[
a_i=a_{i-1}+1+\sum_{j=1}^{i-1}a_j
=a_{i-1}+\sum_{j=0}^{i-1}a_j.
\]
This is recurrence~\ref{rec:5} in Lemma~\ref{lem:fibrec}, so $a_i=F\p{2i+1}$, as required.
\end{proof}

Let $\seq{b_i}_{i\geq0}$ be the sequence where $b_0=0$ and $b_i=a_{i-1}+b_{i-1}$ for all $i>1$.  We have the following proposition.
\begin{prop}\label{prop:sumproperty}
For all $i\geq1$, $a_i=2a_{i-1}+b_{i-1}$.
\end{prop}
\begin{proof}
The recurrence $b_i=a_{i-1}+a_{i-1}$ implies that
\[
b_i=\sum_{j=0}^{i-1}a_j.
\]
By Proposition~\ref{prop:wlen} and recurrence~\ref{rec:3} in Lemma~\ref{lem:fibrec}, we see that $b_i=F\p{2i}$.  By recurrence~\ref{rec:1} in Lemma~\ref{lem:fibrec}, we see that $a_i=2a_{i-1}+b_{i-1}$, as required.
\end{proof}
Recall that our greedy representing sequence had a similar property to that relating $\seq{a_i}$ to $\seq{b_i}$, except that the values in $\seq{b_i}$ sometimes did not change.  Here, we forced these values to change every iteration.  This was equivalent to choosing the recurrence $a_i=3a_{i-1}-a_{i-2}$ at every step.  This consistent choice of recurrence causes our sequences to behave much more predictably.

We will now obtain another characterization of $w$, this time
in terms of morphisms.
First, we need a lemma.
\begin{lemma}\label{lem:phione}
Let $\varphi:\st{0,1}^*\to\st{0,1}^*$ be the morphism defined by $\varphi\p{0}=001$ and $\varphi\p{1}=01$.  Then,
\[
\varphi^n\p{1}=\p{\prod_{i=1}^{n}\varphi^{n-i}\p{0}}1
\]
\end{lemma}
\begin{proof}
The proof is by induction on $n$.  If $n=1$, then we have $\varphi\p{1}=01$, which is true by definition.  Now, assume that
\[
\varphi^{n-1}\p{1}=\p{\prod_{i=1}^{n-1}\varphi^{n-i-1}\p{0}}1.
\]
We have
\begin{align*}
\varphi^n\p{1}&=\varphi\p{\varphi^{n-1}\p{1}}\\
&=\varphi\p{\p{\prod_{i=1}^{n-1}\varphi^{n-i-1}\p{0}}1}\\
&=\p{\prod_{i=1}^{n-1}\varphi\p{\varphi^{n-i-1}\p{0}}}\varphi\p{1}\\
&=\p{\prod_{i=1}^{n-1}\varphi^{n-i}\p{0}}01\\
&=\p{\prod_{i=1}^{n-1}\varphi^{n-i}\p{0}}\varphi^0\p{0}1\\
&=\p{\prod_{i=1}^n\varphi^{n-i}\p{0}}1,
\end{align*}
as required.
\end{proof}
Now we can give our construction.
\begin{prop}\label{prop:hd0l}
Let $\varphi:\st{0,1}^*\to\st{0,1}^*$ be the morphism defined by $\varphi\p{0}=001$ and $\varphi\p{1}=01$.  Let $\psi:\st{0,1}^*\to\st{0,1,2}^*$ be the morphism defined by $\psi\p{0}=01$ and $\psi\p{1}=2$.  Then, $w=\psi\p{\varphi^\omega\p{0}}$.
\end{prop}
\begin{proof}
We claim that $w_i=\psi\p{\varphi^{i-1}\p{0}}$ for all $i$, and our desired conclusion will follow by taking the limit as $i$ goes to $\infty$.  The proof will be by induction on $i$.  If $i=1$, we see that $\psi\p{\varphi^0\p{0}}=\psi\p{0}=01=w_1$, as required.

Now, assume that $w_j=\psi\p{\varphi^{j-1}\p{0}}$ for all $j<i$.  We will show that $w_i=\psi\p{\varphi^{i-1}\p{0}}$.  We have (by induction and by Lemma~\ref{lem:phione})
\begin{align*}
w_i&=w_{i-1}^2\p{\prod_{j=1}^{i-2}w_{i-1-j}}2\\
&=\psi\p{\varphi^{i-2}\p{0}}^2\p{\prod_{j=1}^{i-2}\psi\p{\varphi^{i-2-j}\p{0}}}\psi\p{1}\\
&=\psi\p{\varphi^{i-2}\p{0}^2\p{\prod_{j=1}^{i-2}\varphi^{i-2-j}\p{0}}1}\\
&=\psi\p{\varphi^{i-2}\p{0}^2\varphi^{i-2}\p{1}}\\
&=\psi\p{\varphi^{i-2}\p{001}}\\
&=\psi\p{\varphi^{i-1}\p{0}},
\end{align*}
as required.
\end{proof}
Notice that $\varphi=\varphi_1\circ\varphi_2$, using the notation from Proposition~\ref{prop:sturmmorph}.  Hence, $\varphi$ is Sturmian.  So, we might expect that there is a connection between $w$ and some Beatty sequence.  Sure enough, we have the following (recalling Definition~\ref{def:wyt}).
\begin{prop}\label{prop:zwythup}
Let $n$ be a nonnegative integer.  Then, $w\bk{n}=0$ if and only if $n\in W_U$.
\end{prop}
The proof of Proposition~\ref{prop:zwythup} will involve Zeckendorf representations.
In particular, it will depend on the following result, which we will state without proof.
\begin{prop}\label{prop:zeckwyt}\cite{oeis201}, \cite{oeis1950}
Let $n$ be a nonnegative integer.  If the Zeckendorf representation of $n$ ends in an even number of zeroes, then $n\in W_L$, and if the Zeckendorf representation of $n$ ends in an odd number of zeroes, then $n\in W_U$.
\end{prop}
We will need the following lemma:
\begin{lemma}\label{lem:altzeck}
For all $j\geq1$, the Zeckendorf representation of $a_j-1$ is $\p{10}^{j-1}1$.
\end{lemma}
\begin{proof}
Notice that $\p{10}^{j-1}1$ is a valid Zeckendorf representation.  Also, recall that $a_j-1=F\p{2j+1}-1$.  So, by recurrence~\ref{rec:4} in Lemma~\ref{lem:fibrec}, we see that $\p{10}^{j-1}1$ is the representation of $a_j-1$, as required.
\end{proof}

We are going to draw comparisons between Zeckendorf representations and $\seq{a_i}$-representations.  We have the following proposition.
\begin{prop}\label{prop:oddfibwd}
A ternary sequence is a valid $\seq{a_i}$-representation if and only if every $2$ in the sequence is followed immediately by some number (zero or more) of consecutive ones and then a zero.
\end{prop}
We will use the following lemma in our proof of Proposition~\ref{prop:oddfibwd}:
\begin{lemma}\label{lem:oddfiba}
The $\seq{a_i}$-representation of $a_j-1$ is $21^{j-1}0$.
\end{lemma}
\begin{proof}
We have, using recurrence~\ref{rec:3} of Lemma~\ref{lem:fibrec},
\begin{align*}
a_j-1&=-1+a_{j-1}+F\p{2i}\\
&=-1+a_{j-1}+\sum_{h=0}^{j-1}a_h\\
&=2a_{j-1}+\sum_{h=1}^{j-2}a_h.
\end{align*}
Hence, the first two greedy subtractions will both consist of subtracting $a_{j-1}$, then the remaining greedy subtractions will be single subtractions of $a_{j-h}$ for $1<h<j$.  This gives the desired $\seq{a_i}$-representation.
\end{proof}

We will now prove Proposition~\ref{prop:oddfibwd}.
\begin{proof}
First, we will show that the $\seq{a_i}$-representation of every positive integer is a ternary string with the given property.  Assume for a contradiction that $n$ is the minimal nonegative integer whose $\seq{a_i}$-representation contains a two that is not immediately followed by some number of ones and then a zero.  Clearly, the representation of $n$ contains a $2$.  Since $n$ is a minimal counterexample, Proposition~\ref{prop:cutoff} implies that the highest place in the $\seq{a_i}$-representation of $n$ is a $2$.
Let $j=\ind\p{n}$.  The $\seq{a_i}$-representation of $a_{j+1}-1$ is $21^{j-1}0$ by Lemma~\ref{lem:oddfiba}.  Any violation would mean that $n>a_{j+1}-1$, as a consequence of Proposition~\ref{prop:lexdig}, contradicting the fact that $j=\ind\p{n}$.  Therefore, all $\st{a_i}$-representations have the desired property as ternary strings.

To show that ternary strings with the desired property are all $\st{a_i}$-representations of some nonnegative integer, we will count the number of them of length at most $\ell$.  This is the same as counting the number of length exactly $\ell$ when we allow leading zeroes.  Since these strings will encode the integers from $0$ through $a_\ell-1$, we want to show that there are $a_\ell$ of them.  Let $c\p{\ell}$ denote the number of these strings of length at most $\ell$.  We have the recurrence
\[
c\p{\ell}=2c\p{\ell-1}+\sum_{j=2}^{\ell-1}c\p{\ell-j},
\]
where the two copies of $c\p{\ell-1}$ cover the non-restrictive cases of the first ternary digit being a $0$ or a $1$ and the sum covers the restrictive case of a $2$ by placing $j-2$ ones and a $0$ immediately after the $2$.
This is recurrence~\ref{rec:5} in Lemma~\ref{lem:fibrec}, and the counts satisfy the appropriate initial conditions for $c_\ell=F\p{2\ell+1}$, as required.
\end{proof}

From this point forward, the following definition will be useful:
\begin{defin}
Let $v$ be the $\seq{a_i}$-representation of some nonnegative integer.  A \emph{two-block} in $v$ is a block of consecutive digits of the form $21^m0$ for some $m\geq0$.
\end{defin}

We have the following result about $\seq{a_i}$-representations and Wythoff numbers.
\begin{prop}\label{prop:oddfibwyth}
A nonnegative integer $n$ is in $W_U$ if and only if
$n$ is a nonvolatile zend in the $\seq{a_i}$-representation.
\end{prop}
The proof of Proposition~\ref{prop:oddfibwyth} is quite long, so we have placed it in appendix~\ref{app:oddfibwyth}.  But, we will use the following lemma in that proof and in some later results, so we state and prove it here.
\begin{lemma}\label{lem:oddfibvol}
Let $n$ be a zend in the $\seq{a_i}$-representation.  Then, $n$ is volatile if and only if the $\seq{a_i}$-representation of $n$ ends in a two-block.
\end{lemma}
\begin{proof}\lb
\iffpf{
We will prove the contrapositive of this case.  Assume that the $\seq{a_i}$-representation of $n$ ends in $0$ but not a two-block.  Then, the ternary string obtained by replacing the final $0$ in the representation of $n$ by a $1$ is a valid $\seq{a_i}$-representation.  Therefore, $n$ is non-volatile, as required.
}
{
Assume that the $\seq{a_i}$-representation of $n$ ends in a two-block.  Then, the ternary string obtained by replacing the final $0$ in the representation of $n$ by a $1$ is not a valid $\seq{a_i}$-representation.  Therefore, $n$ is volatile, as required.
}
\end{proof}

We will now prove Proposition~\ref{prop:zwythup}.  (Recall that Proposition~\ref{prop:zwythup} states that $w\bk{n}=0$ if and only if $n\in W_U$.)
\begin{proof}
By Proposition~\ref{prop:oddfibwyth}, $W_U$ is precisely the set of non-volatile zends in the $\seq{a_i}$-representation.  Therefore, since $w$ is the representation word of $\seq{a_i}$, $w\bk{n}=0$ if and only if $n\in W_U$, as required.
\end{proof}

So, we have now arrived at a description of the positions of the zeroes in $w$.  We will now attempt to use Theorem~\ref{thm:absmain} to prove that $w$ is a the Nim sequence of a subtraction set.  Let $T=\st{a_i-1\mid i\geq1}$, and let $I$ be the set containing $1$ and all positive integers whose $\seq{a_i}$-representations end in a two-block.  We have the following.
\begin{lemma}\label{lem:aforb}
We have $\p{W_U+I}\cap W_U=\emptyset$.
\end{lemma}
\begin{proof}
Assume for a contradiction that there exist $m,n\in W_U$ and $a\in I$ such that $m+a=n$.  In other words, $a=n-m$.  Recall that $W_U$ is the Beatty sequence generated by $\phi^2$.  So, by Proposition~\ref{prop:beattyforbdiff}, there exists a positive integer $b$ such that $a=\fl{b\phi^2}$ or $a=\ceil{b\phi^2}$.  Since $2<\phi^2<3$, we cannot have $a=1$.  So, the $\seq{a_i}$-representation of $a$ ends in a two-block.  If $a=\fl{b\phi^2}$, then $a\in W_U$.  But, $I\subset W_L$, so $a\in W_L$, a contradiction.  If $a=\ceil{b\phi^2}$, then $a\geq3$ and $a-1=\fl{b\phi^2}$, so $a-1\in W_U$.  But, $a-1\in W_L$, since $1$ plus a nonvolatile zend (an element of $W_U$) will not end in a two-block (since it will end in $1$).  This is a contradiction.  Therefore, there cannot exist such an $a\in I$, so it follows that $\p{W_U+I}\cap W_U=\emptyset$.
\end{proof}

We can now formulate and prove our main theorem.
\begin{theorem}\label{thm:main}
For all sets $S$ with $T\subseteq S\subseteq I$, $w$ is the Nim sequence for $S$.  Furthermore, for all $i$, $\p{w\left[0..a_i\right)}^\omega$ is the Nim sequence for the set $S_i=\st{s\in S\mid s<a_i}$.  In particular, there exists a subtraction set with positive density in the natural numbers (namely $I$) whose Nim sequence is ternary and aperiodic.
\end{theorem}
\begin{proof}
By Theorem~\ref{thm:absmain}, Lemma~\ref{lem:aforb}, and Lemma~\ref{lem:oddfibvol}, we have that $w$ is the Nim sequence for $S$.  Then, since $a_i>2a_{i-1}$ for all $i$, by Proposition~\ref{prop:2aain}, $\p{w\left[0..a_i\right)}^\omega$ is the Nim sequence for $S_i$, as required.  The fact that $I$ has positive density in the natural numbers follows from the fact that the difference between consecutive terms of $I$ is at most $5$.
\end{proof}

Note that the sequence of positive integers whose $\seq{a_i}$-representations end in a two-block is sequence A089910 in OEIS~\cite{oeis89910}.  The sequence with a $1$ prepended to it was not in OEIS at the time of writing this paper.  Also, note that the Sprague-Grundy values have a simple description in terms of $\seq{a_i}$-representations of the indices.  Hence, the winning strategy for this game can be computed efficiently.

See Appendix~\ref{app:promo} for an extension of Theorem~\ref{thm:main} that gives aperiodic Nim sequences over $\Sigma_k$ for all $k\geq2$.
\section{Open Problems}\label{s:open}
There are a number of problems that remain open relating to aperiodic, bounded Nim sequences of subtraction games.  Many of these problems stem from the fact that we built up the abstract theory of representation words and then applied it only to a single case.  We have the following conjecture:
\begin{conj}\label{conj:ext1}
Let $\seq{a_i}_{i=0}^n$ ($n\geq1$) be a finite representing sequence with representation word $w$.  Let $\seq{b_i}$ be defined as follows:
\[
b_i=\begin{cases}
a_i & i\leq n\\
3b_{i-1}-b_{i-2} & i>n,
\end{cases}
\]
and let $v$ be the representation word of $\seq{b_i}$.  If $w$ is the Nim sequence of $T\p{\seq{a_i}}$, then $v$ is the Nim sequence of $T\p{\seq{b_i}}$.
\end{conj}
In other words, any finite sequence that can be produced by Algorithm~\ref{alg:greedyrep} can be extended to an infinite sequence that yields a bounded, aperiodic Nim sequence.  In particular, the truth of Conjecture~\ref{conj:ext1} would imply that the greedy sequence can be extended forever.  This raises the question of its asymptotic growth rate.  If its terms continue to obey the recurrences $a_i=3a_{i-1}-2a_{i-2}$ or $a_i=3a_{i-1}-a_{i-2}$, then its growth rate is $\Omega\p{2^n}$ and $O\p{\phi^{2n}}$.  For reference, the minimal sequence in Theorem~\ref{thm:main} is $\Theta\p{\phi^{2n}}$.  A related question is whether there exists a representing sequence $\seq{a_i}$ with growth rate $o\p{2^n}$ such that its representing word is aperiodic and the Nim sequence of $T\p{\seq{a_i}}$ (where $\seq{a_i}$ is the minimal representation sequence of that word).

Recall that Theorem~\ref{thm:fergtracegen} gives a way of producing infinite, strongly Fergusonian representation words, and if $p_{i,i-1}\geq2$ for all $i$ in that construction, then those words will be aperiodic.  The ultimate goal is a characterization of what choices of $p_{i,\ell}$ yield Nim sequences.  The case where $p_{i,i-1}=2$ and $p_{i,\ell}=1$ if $\ell<i-1$ corresponds to the sequence in Theorem~\ref{thm:main}.  In general, Conjecture~\ref{conj:ext1} considers extending the collection of $p$ values in a way where $p_{i,i-1}=2$ for all sufficiently large $i$.  But what if we want to extend the sequence in a way where $p_{i,i-1}=m$ for all sufficiently large $i$?  Then, the eventual recurrence we obtain is $a_m=\p{m+1}m_{n-1}+\p{m-1}a_{m-2}$.  But not all sequences can be extended with such a recurrence for $m>2$.  For example, in the case $m=3$, the sequence $\seq{1,2,5}$ would be extended to $\p{1,2,5,16}$, but $\pref\p{\st{1,4,15}}\neq\eps$.  We have the following conjecture.
\begin{conj}\label{conj:ext2}
Let $\seq{a_i}_{i=0}^n$ ($n\geq2$) be a finite representing sequence with representation word $w$, and let $n\geq2$ be an integer.  Let $\seq{b_i}$ be defined as follows:
\[
b_i=\begin{cases}
a_i & i\leq n\\
\p{m+1}b_{i-1}-\p{m-1}b_{i-2} & i>n,
\end{cases}
\]
and let $v$ be the representation word of $\seq{b_i}$.  If $w$ is the Nim sequence of $T\p{\seq{a_i}}$ and if $a_n=\p{m+1}a_{n-1}-\p{m-1}a_{n-2}$, then $v$ is the Nim sequence of $T\p{\seq{b_i}}$.
\end{conj}
In other words, any sequence that can be extended once by such a recurrence can be extended infinitely by the same recurrence.

A specific instance of Conjecture~\ref{conj:ext2} comes from the sequence $\p{1,2,5,13,42}$ with $m=3$.  This construction can be described by letting $w_{-1}=01$, $w_0=01012$, $w_1=0101201012012$, and for all $i\geq2$, let
\[
w_i=w_{i-1}^3\p{\prod_{j=1}^{i-2}w_{i-1-j}}012.
\]
Let
\[
w=\limi{i}{w_i}.
\]
In a similar vein to Proposition~\ref{prop:hd0l}, this construction can be construed as a ternary morphism applied to the fixed point of a binary morphism.  In this case, the binary morphism maps $0$ to $0001$, and it maps $1$ to $01$.  This is not a Sturmian morphism.  Hence, the positions of the zeroes in the representation word will not be described by a Beatty sequence, so that machinery will be of little use in solving the problem even in this specific case.

Another observation is that Proposition~\ref{prop:oddfibwyth} depends very heavily on the sequences begin Fibonacci numbers and odd-index Fibonacci numbers.  It seems natural that there should be some relationship between representations by a sequence and representation by a subsequence of that sequence, especially when the subsequence is defined by a simple pattern (such as odd indices).  Any sort of generalization to Proposition~\ref{prop:oddfibwyth} would be a great step toward perhaps solving the general $p_{i,\ell}$ problem.

Other open problems involve removing or weakening extra conditions we needed on our theorems.  For example, Theorem~\ref{thm:absmain} has the strange condition for all $a<a_i$ with $a\in N$, $a+a_i-S_i\subseteq L$ that we needed to prove that we can extract Nim sequences of finite subtraction sets, given that the infinite case works.  Proposition~\ref{prop:2aain} gives us a method to verify this condition that applies in most cases we would consider.  In the one other case (Proposition~\ref{prop:nodoub}), we were able to verify the condition directly.  So, we have the following two questions:
\begin{itemize}
\item Is this condition necessary to allow the infinite case of Theorem~\ref{thm:absmain} to be extended to the finite case?
\item Is this condition automatically satisfied by any parameters satisfying the infinite case of Theorem~\ref{thm:absmain}?
\end{itemize}
The answer to the first question may very well be yes, but we do not have any example that proves this.  The most likely form of an example would be that the finite Nim sequences have nontrivial prefixes.

Another question related to Theorem~\ref{thm:absmain} is when we can take $I=V$.  In Theorem~\ref{thm:main}, we were able to take $I=V$, but the sequence $\seq{1,2,5,8,16,27,57,87}$ does not allow this.  It is open to give a characterization of when $I=V$ works and when it does not.

It is also applicable to ask how much of the work of this paper can be automated.  In recent years, computers have made quick work of numerous theorems whose original human proof were quite difficult.  Much of the work of this paper feels somewhat mechanical, but, in the end, the fact that Proposition~\ref{prop:oddfibwyth} is specific to Zeckendorf representations was critical to the main result.

One last observation is that the sequences $W_L$ and $W_U$ are used to describe the $P$-positions in Wythoff's Game.  The full Sprague-Grundy function for this game is not known; perhaps the work of this paper inadvertantly serves as progress toward that long-standing problem.
\section{Acknowledgements}
I would like to thank Dr. Doron Zeilberger of Rutgers University for providing me with useful feedback on this work and on the paper and for teaching the course that spawned the idea for this research.  I would also like to thank Nathaniel Shar of Rutgers University for proofreading a draft of this paper and providing me with useful feedback.
%
%
%
%
%
\appendix
\section{Routine Proofs about Nim Sequences}\label{app:nsrp}
To prove Proposition~\ref{prop:sinf}, we will use the following lemma.
\begin{lemma}\label{lem:maxsg}
Let $S$ be a subtraction set.  All of the Sprague-Grundy values for $S$ are at most $\ab{S}$.
\end{lemma}
The proof of Lemma~\ref{lem:maxsg} is easy, so we omit it.
We can now prove Proposition~\ref{prop:sinf}.
\begin{proof}
We will prove the contrapositive.  Let $S$ be a finite set of natural numbers.  We will show that the Nim sequence for $S$ is eventually periodic.  Let $s$ denote the maximum value in $S$, and let $k$ denote the cardinality of $S$.  Construct the directed graph $G$ whose vertices are $s$-tuples of integers in $\st{0,1,\ldots, k}$ and where each vertex has exactly one outgoing edge defined by the following rule:
\[
\p{a_1,a_2,\ldots,a_s}\to\p{a_2,a_3,\ldots,a_s,\begin{array}{c}\mex\vspace{-0.1 in}\\\hspace{0in}_{t\in S}\end{array}a_{s+1-t}}.
\]
By Lemma~\ref{lem:maxsg}, every block of $s$ consecutive terms in the Nim sequence form a vertex in $G$.  Also, by our construction of edges, computing terms after the $s^{th}$ in the Nim sequence for $S$ correspond to walking along edges in $G$.  By the Pigeonhole Principle and the structure of $G$, any such walk must eventually end up in a cycle.  Such a cycle corresponds to an eventual period in the Nim sequence, as required.
\end{proof}

We now give the proof of Proposition~\ref{prop:gcd}.
\begin{proof}
Let $n=qg+r$ for integers $q$ and $0\leq r<g$.  If $T_q$ is the set of positions that can be legally reached in one move from position $q$ in the game with subtraction set $S$, then $gT_q+r$ is the set of positions that can be legally reached in one move from position $n$ in the game with subtraction set $gS$.  In particular, the remainder $r$ is fixed throughout the game.  Hence, each sequence with jumps of size $g$ is the same as the Nim sequence for $S$.  This suffices to prove the desired statement, as there are $g$ such sequences interleaved in the Nim sequence of $gS$.
\end{proof}

Finally, we give the proof of Corollary~\ref{cor:ternary}.
\begin{proof}
Let $\st{a_i}_{i\geq0}$ be a binary sequence that is the Nim sequence of some subtraction game.  Let $s=\min S$.  We will show by induction that $a_i=\fl{\frac{i}{s}}\p{\modd 2}$.  First, we see that $a_i=0$ for all $i<s$ since there are no legal moves, as required.  Now, assume that $a_{i-s}=\fl{\frac{i-s}{s}}\p{\modd 2}$.  We will show that $a_i\neq a_{i-s}$, which is our desired conclusion.  If $a_{i-s}=0$, then $a_i=1$ by Ferguson's Theorem, as required.  If $a_{i-s}=1$, then $a_i\neq1$ by Ferguson's Theorem.  Since $a_i=0$ or $a_i=1$, then $a_i=0$, as required.  Hence, we have shown that the sequence $\st{a_i}_{i\geq0}$ is periodic with period $0^s1^s$.
\end{proof}
\section{Algorithm for Finding Period and Prefix}\label{app:alg}
In this appendix, we give our algorithm for finding the period and prefix of a finite subtraction set.
\begin{algorithm}
\caption{$GET\_PERIOD\_AND\_PREFIX\p{S}$}
\label{alg:gpp}
\begin{algorithmic}
\REQUIRE $S$ finite subtraction set
\STATE $n\gets 50$ \COMMENT{the number $50$ here is arbitrary}
\WHILE{\TRUE}
	\STATE Let $w$ be the first $n$ terms of the Nim sequence of $S$
	\FOR{$i$ \textbf{from}  $n-1$ \textbf{down to} $2\cdot\max\p{S}$}
		\STATE $v\gets w\left[i..n\right)$
		\STATE $x\gets v^{\ceil{\frac{2\cdot\max S}{\ab{v}}}+1}$
		\IF{$w=ux$ for some $u$}
			\STATE $ok\gets\TRUE$
			\FOR{$j$ \textbf{from} $\max\p{S}$ \textbf{to} $\ab{x}$}
				\IF{$x\bk{j}\neq\begin{array}{c}\mex\vspace{-0.1 in}\\\hspace{0in}_{s\in S}\end{array}x\bk{j-s}$}
					\STATE\COMMENT{$x\neq\per\p{S}$}
					\STATE $ok\gets\FALSE$
					\STATE \textbf{break}
				\ENDIF
			\ENDFOR
			\IF{ok}
				\STATE\COMMENT{some cyclic shift of $v$ equals $\per\p{S}$}
				\WHILE{$\ab{u}>0$ \AND $u\bk{\ab{u}-1}=v\bk{\ab{v}-1}$}
					\STATE\COMMENT{$u$ not yet minimal}
					\STATE $v\gets v\bk{\ab{v}-1}v\left[0..\ab{v}-1\right)$
					\STATE $u\gets u\left[0..\ab{u}-1\right)$
				\ENDWHILE
				\RETURN $\pref\p{S}=u$ \AND $\per\p{S}=v$
			\ENDIF
		\ENDIF
	\ENDFOR
	\STATE $n\gets2n$
\ENDWHILE
\end{algorithmic}
\end{algorithm}
\FloatBarrier
We will now prove that Algorithm~\ref{alg:gpp} is correct.  This will follow from proving the following claims:
\begin{itemize}
\item When $ok$ is true after the inner \texttt{for} loop terminates, $v$ is an eventual period of the Nim sequence of $S$ of minimal length.
\item The final $u$ and $v$ have minimal length among all decompositions $w=uv^\omega$, where $w$ is the Nim sequences of $S$.
\item This algorithm terminates.
\end{itemize}
We will prove each of these items in turn.
\begin{description}
\item[$v$ minimal length period after inner \texttt{for} loop:] Note that all eventual periods of the Nim sequence of $S$ are powers of some minimal period.  Since we test periods in increasing order of length, it will suffice to show that $v$ is an eventual period of the Nim sequence of $S$.  This follows from the fact that $w$ ends in sufficiently many copies of $v$ that future terms of the Nim sequence can be determined only from terms in these copies, and these copies agree with $v$ being an eventual period.
\item[$u$ and $v$ minimal with respect to decomposition $uv^\omega$:] We know that $v$ has minimal length, as no operation after the inner \texttt{for} loop changes the length of $v$.  It remains an eventual period, as any cyclic permutation of an eventual period is an eventual period.  Also, $u$ has minimal length among all potential prefixes by the stopping condition on that loop.
\item[Termination:] The only loop that appears that it may not terminate is the outermost \texttt{while} loop.  But, if an eventual period exists, this algorithm will find it and this loop will terminate.  Since $S$ is finite, the Nim sequence has an eventual period by Proposition~\ref{prop:sinf}.  Therefore, Algorithm~\ref{alg:gpp} terminates, as required.
\end{description}
\section{Proof of Theorem~\ref{thm:fergtracegen}}\label{app:fergtracegen}.
\begin{proof}
We will show that, for all $n$, if $\ab{w_i}\leq n<\ab{w_{i+1}}$
\[
w\bk{n}=\begin{cases}
n & n<k\\
k & n=\ab{w_{i+1}}-1\\
w\bk{n-\ab{w_i}} & \text{otherwise}
\end{cases}
\]
This will show, by Theorem~\ref{thm:repwdtrace}, that $w$ is a representation word and that its representation sequence is $\seq{1,k=\ab{w_1},\ab{w_2},\ab{w_3},\ldots}$.

The first case is clear from the definition of $w_1$, and the second case follows from our construction where we place a $k$ at the end of each $w_i$ for $i\geq2$.  So, it will suffice to prove the last case.  If $\ab{w_i}\leq n<\p{p_{i,i-1}-1}\ab{w_{i-1}}+\ab{w_i}$, then we are in some copy of $w_{i-1}$ at the beginning of $w_i$ other than the first copy.  Going back $\ab{w_i}$ positions in $w_i$ will put us in the same position in the previous copy of $w_{i-1}$, so the proposition holds for $\ab{w_i}\leq n<\p{p_{i,i-1}-1}\ab{w_{i-1}}+\ab{w_i}$.
Now, it will suffice to show that the subword consisting of the remaining values of $n$ (not including $\ab{w_{i+1}}-1$), which equals
\[
\prod_{j=2}^{i-1}w_{i-j}^{p_{i,i-j}}
\]
for a given $i$, is a prefix of $w_{i-1}$.
We will actually prove a stronger statement: For all $\ell$ with $1\leq\ell\leq i-1$,
\[
\prod_{j=\ell+1}^{i-1}w_{i-j}^{p_{i,i-j}}
\]
is a prefix of $w_{i-\ell}$.  We will prove this statement by induction on $\ell$.  If $\ell=i-1$, then the product is empty, so the statement is true.  Now, fix $\ell$ and assume the statement holds for all larger values of $\ell$.  We have
\begin{align*}
w_{i-\ell}&=\p{\prod_{j=1}^{i-\ell-1}w_{i-\ell-j}^{p_{i-\ell, i-j-\ell}}}k\\
&=\p{\prod_{j=\ell+1}^{i-1}w_{i-j}^{p_{i-\ell, i-j}}}k.
\end{align*}
If $p_{i-\ell, i-j}=p_{i,i-j}$ for all $1\leq j<i-\ell$, then we have that
\begin{align*}
w_{i-\ell}&=\p{\prod_{j=\ell+1}^{i-1}w_{i-j}^{p_{i-\ell, i-j}}}k\\
&=\p{\prod_{j=\ell+1}^{i-1}w_{i-j}^{p_{i, i-j}}}k,
\end{align*}
which has the desired prefix.  Now, assume that, for some $j$, $p_{i-\ell, i-j}\neq p_{i,i-j}$.  This means that $p_{i-\ell, i-j}>p_{i,i-j}$.  Let $\ell_0$ denote the smallest $j$ where this happens.  By our inductive hypothesis, $w_{i-\ell_0}$ has
\[
\prod_{j=\ell_0+1}^{i-1}w_{i-j}^{p_{i,i-j}}
\]
as a prefix, so we can write
\[
w_{i-\ell_0}=\p{\prod_{j=\ell_0+1}^{i-1}w_{i-j}^{p_{i,i-j}}}v
\]
for some word $v$.  So, we have
\begin{align*}
w_{i-\ell}&=\p{\prod_{j=\ell+1}^{i-1}w_{i-j}^{p_{i-\ell, i-j}}}k\\
&=\p{\prod_{j=\ell+1}^{\ell_0-1}w_{i-j}^{p_{i-\ell, i-j}}}w_{i-\ell_0}^{p_{i-\ell, i-\ell_0}}\p{\prod_{j=\ell_0+1}^{i-1}w_{i-j}^{p_{i-\ell, i-j}}}k\\
&=\p{\prod_{j=\ell+1}^{\ell_0-1}w_{i-j}^{p_{i, i-j}}}w_{i-\ell_0}^{p_{i, i-\ell_0}}w_{i-\ell_0}w_{i-\ell_0}^{p_{i-\ell, i-\ell_0}-p_{i,i-\ell_0}-1}\p{\prod_{j=\ell_0+1}^{i-1}w_{i-j}^{p_{i-\ell, i-j}}}k\\
&=\p{\prod_{j=\ell+1}^{\ell_0-1}w_{i-j}^{p_{i, i-j}}}w_{i-\ell_0}^{p_{i, i-\ell_0}}\p{\prod_{j=\ell_0+1}^{i-1}w_{i-j}^{p_{i,i-j}}}\cdot vw_{i-\ell_0}^{p_{i-\ell, i-\ell_0}-p_{i,i-\ell_0}-1}\p{\prod_{j=\ell_0+1}^{i-1}w_{i-j}^{p_{i-\ell, i-j}}}k\\
&=\p{\prod_{j=\ell+1}^{i-1}w_{i-j}^{p_{i, i-j}}}vw_{i-\ell_0}^{p_{i-\ell, i-\ell_0}-p_{i,i-\ell_0}-1}\p{\prod_{j=\ell_0+1}^{i-1}w_{i-j}^{p_{i-\ell, i-j}}}k,
\end{align*}
which has the desired prefix.

Now, all that remains is to show that $w$ is strongly Fergusonian.  By Proposition~\ref{prop:traceferg}, it will suffice to show that every $k$ in $w$ follows a $k-1$.  We will prove this by induction on $i$.  There are no $k$'s in $w_1$, so the claim is true for $w_1$.  Now, assume that the claim holds for all $w_j$ with $j<i$.  We have
\[
w_i=\p{\prod_{j=1}^{i-1}w_{i-j}^{p_{i,i-j}}}k.
\]
All $k$'s, other than the last one, follow $\p{k-1}$'s by the inductive hypothesis.  The last $k$ also follows a $k-1$, since $p_{i,1}>0$ and $w_1$ ends in $k-1$.  Therefore, every $k$ in $w_i$ follows a $k-1$, as required.
\end{proof}
\section{Proof of Proposition~\ref{prop:oddfibwyth}}\label{app:oddfibwyth}
\begin{proof}
By Lemma~\ref{lem:oddfibvol}, it will suffice to show that there exists a value-preserving bijection between Zeckendorf representations ending in an odd number of zeroes (elements of $W_U$ by Proposition~\ref{prop:zeckwyt}) and $\seq{a_i}$-representations ending in $0$ and not ending in a two-block.  We will give one direction of the bijection via an algorithm; the other direction will follow from a counting argument.
\begin{algorithm}
\caption{$ODD\_INDEX\_FIB\_TO\_ZECK\p{\seq{d_i}}$}
\label{alg:ofzeck}
\begin{algorithmic}
\REQUIRE $\seq{d_i}$ is the $\seq{a_i}$-reprsentation of some nonnegative integer $n$, given as an infinite ternary sequence (with only finitely many nonzero terms)
\REQUIRE $d_0=0$
\REQUIRE $\seq{d_i}$ does not end in a two-block (when treated as a digital representation)
\STATE Let $\seq{e_i}$ be the ternary sequence satisfying
\begin{itemize}
\item$e_0=d_0$
\item$e_{2i-1}=d_i$ for all $i\geq1$
\item$e_{2i}=0$ for all $i\geq1$.
\end{itemize}
\COMMENT{$\seq{e_i}$ is the result of inserting zeroes between every pair of digits in $\seq{d_i}$ other than the least significant pair}
\INVARIANT $\sum\limits_{i=0}^\infty e_iF\p{i+2}=n$
\INVARIANT $\seq{e_i}$ ends with an odd number of zeroes (when treated as a digital representation)
\STATE Let $j$ be the index of the most significant nonzero digit in $\seq{e_i}$.
\FOR {$i$ \textbf{from} $j$ \textbf{down} \TO $2$}
   \IF {$e_i=2$}
	\STATE $e_i\gets0$
	\STATE $e_{i+1}\gets e_{i+1}+1$
	\STATE $e_{i-2}\gets e_{i-2}+1$
   \ENDIF
\ENDFOR
\INVARIANT $\seq{e_i}$ contains only zeroes and ones
\WHILE {$\seq{e_i}$ contains a pair of consecutive ones}
\STATE Let $j$ be the index of the most significant nonzero digit in $\seq{e_i}$. representation)
\FOR {$i$ \textbf{from} $j$ \textbf{down} \TO $1$}
   \IF {$e_i=1$ \AND $e_{i-1}=1$}
	\STATE $e_i\gets0$
	\STATE $e_{i-1}\gets 0$
	\STATE $e_{i+1}\gets e_{i+1}+1$
   \ENDIF
\ENDFOR
\ENDWHILE
\ENSURE $\seq{e_i}$ is a valid Zeckendorf representation
\RETURN $\seq{e_i}$
\end{algorithmic}
\end{algorithm}

Algorithm~\ref{alg:ofzeck} is an adaptation of Peter Fenwick's algorithm for adding numbers using their Zeckendorf representations~\cite{fenwick}.  In that vein, we will refer to the operation performed by the first \texttt{for} loop as a \emph{two-removal}, and we refer to the operation performed by the second \texttt{for} loop as an \emph{one-removal}.  We claim that Algorithm~\ref{alg:ofzeck} converts an $\seq{a_i}$-representation ending in $0$ but not a two-block to a Zeckendorf representation for the same number, and the Zeckendorf representation will end in an odd number of zeroes.  In the algorithm, we have designated three invariants and one postcondition.  When we declare something to be an invariant, we mean that it is a loop invariant of all future loops in the algorithm (included nested loops).  To prove the correctness of this algorithm, we must ensure that no illegal operations are performed, we must check the three invariants and the postcondition in the algorithm, and we must verify that this algorithm terminates.  Note that any of these statements that is an invariant of the second \texttt{for} loop is automatically an invariant of the \texttt{while} loop.  Hence, we will only check the invariants for the \texttt{for} loops.

Before we do check these facts, we will prove a useful property of the first \texttt{for} loop, which we will call the \emph{even zero property}. We claim that for all $m\geq0$, whenever $i=2m$ in that loop, $e_i=0$.  The proof will be by induction on $j-2m$.  Notice that $j$ must be odd.  If $j-2m=1$, the claim is true, since this loop can never modify $e_{j-1}$.  Now, assume that for all $m'>m$, whenever $i=2m'$ in that loop, $e_i=0$.  Assume for a contradiction that $e_i\neq0$ when $i=2m$.  This means that $e_i$ must have been modified by a prior iteration of the loop.  The only iteration that could have made such a modification was the iteration where $i=2m+2=2\p{m+1}$.  This implies that $e_{2\p{m+1}}=2$ in the iteration where $i=2\p{m+1}$, a contradiction.

We will now check all of the desired properties.
\begin{description}
\item[No Illegal Operations:] The only illegal operation that could conceivably take place during this algorithm's execution would be assigning to $e_i$ for some $i<0$.  But, we indexed our \texttt{for} loops so that this would not happen.
\item[Invariant 1: $\sum\limits_{i=0}^\infty e_iF\p{i+2}=n$:] We must show that this is an invariant of both \texttt{for} loops.
\begin{description}
\item[First \texttt{for} loop:] To check that this invariant holds at initialization of this loop, it will suffice to check that
\[
\sum_{i=0}^\infty e_iF\p{i+2}=n
\]
when we defined $\seq{e_i}$.  Fix $\seq{e_i}$ to be its initial value.  We have that
\begin{align*}
n&=\sum_{i=0}^\infty d_ia_i\\
&=\sum_{i=0}^\infty d_iF\p{2i+1}\\
&=e_0F\p{1}+\sum_{i=1}^\infty e_{2i-1}F\p{2i+1}\\
&=e_0F\p{2}+\sum_{i=1}^\infty e_{2i-1}F\p{2i+1}+\sum_{i=1}^\infty e_{2i}F\p{2i+2}\\
&=\sum_{i=0}^\infty e_iF\p{i+2},
\end{align*}
as required.

Now, we must check that the invariant holds after every iteration of the loop.  This loop performs two-removals, which preserve the invariant as a consequence of recurrence~\ref{rec:0} of Lemma~\ref{lem:fibrec}.  Therefore, this invariant holds for the first \texttt{for} loop.
\item[Second \texttt{for} loop:] Since the invariant holds at the end of the final iteration of the first \texttt{for} loop, it must hold at the initialization of the second \texttt{for} loop.  Now, we must check that the invariant holds after every iteration of the loop.  This loop performs one-removals, which preserve the invariant as a consequence of the fundamental Fibonacci recurrence ($F\p{n}=F\p{n-1}+F\p{n-2}$).  Therefore, this invariant holds for the second \texttt{for} loop.
\end{description}
\item[Invariant 2: $\seq{e_i}$ ends with an odd number of zeroes:] We must show that this is an invariant of both \texttt{for} loops.
\begin{description}
\item[First \texttt{for} loop:] Let $i$ be the smallest positive integer such that $d_i>0$.  By the preconditions that $d_0=0$, we have $i>0$.  Hence, $e_{2i+1}$ is the nonzero term with the smallest index in that sequence (in the original sequence $\seq{e_i}$).  In particular, $\seq{e_i}$ ends with an odd number ($2i+1$) of zeroes (when treated as a digital representation).  Therefore, the invariant holds at the initialization of the first \texttt{for} loop.

Now, we must check that the invariant holds after every iteration of the loop.
Any iteration where the loop makes no assignments clearly preserves the invariant, so assume with loss of generality that we are removing a two in position $i$.  In particular, $e_i\neq0$.  We must have $e_{i-1}=0$ by the even zero property.  So, if all positions after $i$ are zero, then we are reducing the number of zeroes at the end of the representation by $2$, thereby preserving the parity.  If some position after $i$ is nonzero, then we do not change the number of zeroes at the end.  In either case, we preserve the parity of the number of trailing zeroes in the representation, as required.
\item[Second \texttt{for} loop:] Since the invariant holds at the end of the final iteration of the first \texttt{for} loop, it must hold at the initialization of the second \texttt{for} loop.  Now, we must check that the invariant holds after every iteration of the loop.  This loop performs one-removals.  If the ones being removed are the final nonzero digits in the representation, then the number of zeroes increases by $2$, thereby preserving the parity.  Otherwise, the number of final nonzero digits remains unchanged.  In either case, we preserve the parity of the number of trailing zeroes in the representation, as required.
\end{description}
\item[Invariant 3: $\seq{e_i}$ contains only zeroes and ones:] We must show that this is an invariant of the second \texttt{for} loop.  To show that it holds at initialization, we must show that it holds after the last iteration of the first \texttt{for} loop.  Notice that the even zero property implies that, after the iteration of that loop when $i=m$, we have $e_\ell\in\st{0,1}$ for all $\ell\geq m$.  The last iteration is when $m=2$.  So, it will suffice to check that $e_0\neq2$ and $e_1\neq2$ at the end of the final iteration of the first \texttt{for} loop.  Notice that $e_0=0\neq2$ by Invariant~2.  Hence, we must check that $e_1\neq2$.  Assume for a contradiction that $e_1=2$.  This was not originally the case, as we assumed that $\seq{d_i}$ did not end in a two-block.  So, some step modified this position.  In particular, some iteration of the first \texttt{for} loop made changes to the sequence, which means that there was at least one $2$ in the original representation.  Since $\seq{d_i}$ did not end in a two-block, there is some $m>0$ such that $d_m=0$.  Let $m$ be the minimal such index.  Notice that $e_{2m-1}=0$ initially.

We will now show that for all $m'\geq1$, $e_{2m'-1}=2$ at the start of the iteration where $i=2m'-1$.  This will contradict the fact that $e_{2m-1}$ could only have been incremented at most once prior to the iteration where $i=2m-1$.  The proof will be by induction on $m'$.  We know that $e_1=2$ in iteration $i=1$ by assumption.  Now, assume that $e_{2m'-3}=2$ at the start of iteration $i=2m'-3$.  Since $\seq{d_i}$ did not end in a two-block, $e_{2m'-3}$ was not initially a $2$.  Hence, it must have been incremented, which means that $e_{2m'-1}=2$ at the start of iteration $i=2m'-1$, as required.
\item[Postcondition: $\seq{e_i}$ is a valid Zeckendorf representation:] Invariant~3 implies that $\seq{e_i}$ contains only zeroes and ones, and the fact that we exited the \texttt{while} loop implies that there are no consecutive ones.  Hence, $\seq{e_i}$ is a valid Zeckendorf representation.
\item[Termination:] The only part of this algorithm that has any chance not to terminate is the \texttt{while} loop.  We see that at the end of every iteration of the \texttt{while} loop, the sequence is lexicographically later than it was at the start.  This, along with Invariant~1 (which gives an upper bound on how high the first $1$ can appear) and Invariant~3 (which restricts the digits that can appear), implies that this loop terminates.
\end{description}

We have now seen how to transform an $\seq{a_i}$-representation ending in $0$ but nonz a two-block into a Zeckendorf representation ending in an odd number of zeroes.  We will now show that the number of $\seq{a_i}$-representations of length at most $\ell$ ending in zero but not a two-block is the same as the number of Zeckendorf representations of length at most $2\ell-1$ ending in an odd number of zeroes.  Notice that this is equivalent to showing that the number of $\seq{a_i}$-representations of length at most $\ell-1$ equals the number of Zeckendorf representations of length at most $2\ell-2$  ending in an even number of zeroes, since in both cases we can chop off the final zero.  Recall that the number of $\seq{a_i}$-representations of length at most $\ell-1$ is $F\p{2\ell-1}$.  Let $d\p{n}$ equal the number of Zeckendorf representations of length at most $2n$ ending in an even number of zeroes.  We have the recurrence
\[
d\p{n}=F\p{2n}+d\p{n-1}.
\]
The first term comes from the strings that end in $1$.  They must end in $01$, and the string preceding the $01$ can be any of the $F\p{2n}$ Zeckendorf representations of length $2n-2$.  The second term comes from the strings that end in $0$.  They must end in $00$, since these strings end in an even number of zeroes.  The string preceding the $00$ can be any of the $d\p{n-1}$ Zeckendorf representations of length $2n-2$ ending in an even number of zeroes.  We claim that $d\p{n}=F\p{2n+1}$.
If $n=1$, there are $2=F\p{3}$ Zeckendorf representations of length at most $2$  ending in an even number of zeroes ($0$ and $1$).  Hence, we have the fundamental Fibonacci recurrence with the same initial conditions.  So, $d\p{n}=F\p{2n+1}$.  
This implies that $d\p{2\ell-2}=F\p{2\ell-1}$, so the number of $\seq{a_i}$-representations of length at most $\ell-1$ equals the number of Zeckendorf representations of length at most $2\ell-2$  ending in an even number of zeroes, as required.
\end{proof}
\section{An Extension of Theorem~\ref{thm:main} to Larger Alphabets}\label{app:promo}
In this appendix, we will show how to build related representation words on larger alphabets from representation words with smaller alphabets.  If the original word was the Nim sequence of a subtraction game, the new word our method creates will also be.
To begin, we will introduce a method of deriving one representing sequence from another, where the new sequence has similar combinatorial properties to the old sequence. 
This will be nearly analogous to moving from binary representations to ternary representations.
\begin{defin}\label{def:promo}
Let $\seq{a_i}$ be a representing sequence, and let $j$ be a positive integer.
The \emph{$j$-promotion} of $\seq{a_i}$ is the sequence $\seq{b_i}_{i\geq0}$ defined recursively by
\[
b_i=
\begin{cases}
a_i & i<j\\
1+\sum\limits_{\ell=0}^{j-2}d_\ell b_\ell+\p{d_{j-1}+1}b_{j-1}+\sum\limits_{\ell=j}^{i-1}d_\ell b_\ell & d_{j-1}>0\\
1+\sum\limits_{\ell=0}^{i-1}d_\ell b_\ell & \text{otherwise},
\end{cases}
\]
where $d_{i-1}d_{i-2}\cdots d_0$ is the $\seq{a_h}$-representation of $a_i-1$.
\end{defin}
For example, the $1$-promotion of the powers of $2$ is the sequence starting with $1$ and then containing $3$ times the powers of $2$.  The powers of $3$ (corresponding to ternary representations) are obtained from the powers of $2$ (corresponding to binary representations) if each digit is promoted once.  Also, the $1$-promotion of the Fibonacci numbers (starting with $1$, $2$) is the Lucas numbers  (starting with $1$, $3$).  We will mainly be concerned with $1$-promotions going forward.

We have another definition related to promotions.
\begin{defin}\label{def:promofun}
Let $\seq{a_i}$ be a representing sequence, and let $j$ be a positive integer.  Let $\seq{b_i}_{i\geq0}$ be the $j$-promotion of $\seq{a_i}$.  The \emph{$j$-promotion function} of $\seq{a_i}$ is the function $f$ on the nonnegative integers satisfying
\[
f\p{n}=
\begin{cases}
n & n < a_{j-1}\\
\sum\limits_{\ell=0}^{j-2}d_\ell b_\ell+\p{d_{j-1}+1}b_{j-1}+\sum\limits_{\ell=j}^{i}d_\ell b_\ell & d_{j-1}>0\\
\sum\limits_{\ell=0}^{i}d_\ell b_\ell & \text{otherwise},
\end{cases}
\]
where $d_{i}d_{i-1}\cdots d_0$ is the $\seq{a_h}$-representation of $n$.  If $m=f\p{n}$, we say that $m$ is the \emph{$j$-promotion} of $n$.
\end{defin}
Under this definition, $f\p{a_i}=b_i$ for all $i$.

We will now use $1$-promotions to expand the alphabet of a representation word to yield another representation word.  We have the following definition.
\begin{defin}\label{def:promorph}
Let $k$ be an integer greater than $1$.  We define a morphism $\varphi_k:\Sigma_k\to\Sigma_{k+1}$ by
\[
\varphi_k\p{i}=\begin{cases}
01 & i=0\\
i+1 & i > 0.
\end{cases}
\]
\end{defin}
The result is as follows.
\begin{theorem}\label{thm:tracepromo}
Let $\seq{a_i}$ be a representing sequence with $a_1=k$, and let $\seq{b_i}$ be the $1$-promotion of $\seq{a_i}$.  Let $w$ be the representation word of $\seq{a_i}$, and let $v$ be the representation word of $\seq{b_i}$.  Then, $\varphi_k\p{w}=v$.
\end{theorem}
\begin{proof}
Let $u=\varphi_k\p{w}$.  We must show that $u=v$.  Let $f$ be the $1$-promotion function of $\seq{a_i}$.  For each index $n$ in $u$, let $g\p{n}$ be the index in $w$ such that $u\bk{n}$ is part of the image of $w\bk{g\p{n}}$.  Also, for each index $n$ in $v$, let
\[
h\p{n}=\begin{cases}
m\text{ such that }f\p{j}=n & \text{if such an }m\text{ exists}\\
m\text{ such that }f\p{j}=n-1 & \text{otherwise}.
\end{cases}
\]
We make the following observations about $f$, $g$, and $h$.
\begin{enumerate}
\item\label{it:tpf} For positive integers $m$,
\[
f\p{m}=\begin{cases}
f\p{m-1}+1 & w\bk{m-1}>0\\
f\p{m-1}+2 & w\bk{m-1}=0.
\end{cases}
\]
This follows from the definition of $1$-promotion function and the fact that the zeroes in $w$ occur precisely in positions that are non-volatile zends in the $\seq{a_i}$-representation.
\item\label{it:tpg} For positive integers $n$,
\[
g\p{n}=\begin{cases}
g\p{n-1}+1 & u\bk{n-1}>0\\
g\p{n-1} & u\bk{n-1}=0.
\end{cases}
\]
This follows from the fact that $w\bk{g\p{n-1}}$ maps to a single symbol unless $w\bk{g\p{n-1}}=0$, in which case it maps to $01$.
\item\label{it:tph} For positive integers $n$,
\[
h\p{n}=\begin{cases}
h\p{n-1}+1 & v\bk{n-1}>0\\
h\p{n-1} & v\bk{n-1}=0.
\end{cases}
\]
If $v\bk{i-1}=0$, then the $\st{b_i}$-representation of $n$ ends in $1$.  Such a number has no preimage under $f$, so we see that $h\p{n}=h\p{n-1}$ in that case.  If $v\bk{n-1}>0$, then $n$ has a preimage under $f$, so $h\p{n}=h\p{n-1}+1$, as required.
\item\label{it:tpfg} For nonnegative integers $n$,
\[
f\p{g\p{n}}=\begin{cases}
n & u\bk{n}\neq1\\
n-1 & u\bk{n}=1.
\end{cases}
\]
We will prove this by induction on $n$.  If $n=0$, then $g\p{n}=0$, and $f\p{g\p{n}}=0$, as required.  Now, assume that this holds for smaller values of $n$.  If $u\bk{n}=1$, then
\begin{align*}
f\p{g\p{n}}&=f\p{g\p{n-1}}\\
&=n-1,
\end{align*}
as required.  If $u\bk{n-1}=1$, then
\begin{align*}
f\p{g\p{n}}&=f\p{g\p{n-1}+1}\\
&=f\p{g\p{n-1}}+1\\
&=n-2+2\\
&=n,
\end{align*}
(since $w\bk{g\p{n-1}}=0$) as required.  If $u\bk{n}\neq1$ and $u\bk{n-1}\neq1$, then
\begin{align*}
f\p{g\p{n}}&=f\p{g\p{n-1}+1}\\
&=f\p{g\p{n-1}}+1\\
&=n-1+1\\
&=n,
\end{align*}
(since $w\bk{g\p{n-1}}\neq0$) as required.
\item\label{it:tpfh} For nonnegative integers $n$,
\[
f\p{h\p{n}}=\begin{cases}
n & v\bk{n}\neq1\\
n-1 & v\bk{n}=1.
\end{cases}
\]
The same proof works here as in~\ref{it:tpfg}, except $u$ is replaced by $v$ and $g$ is replaced by $h$.
\end{enumerate}
We will now prove that $u=v$.  We will show that, for each value $\ell$ with $0\leq\ell\leq k+1$, $u$ and $v$ take that value in the same places.  We will go by induction on $\ell$.  We begin with the case $\ell=1$.  By~\ref{it:tpfg}, if $n$ has no preimage under $f$, $u\bk{n}=1$ (since $f\p{g\p{n}}$ cannot equal $n$).  Similarly, by~\ref{it:tpfh}, if $n$ has no preimage under $f$, $v\bk{n}=1$.  If, on the other hand, $n$ has a preimage under $f$, there will be some $n'$ such that $f\p{g\p{n'}}=n$, since~\ref{it:tpg} implies that $g$ is onto the nonnegative integers.  By~\ref{it:tpfg}, this implies that $n'=n$ or $n'=n+1$.  If $n'=n$, then $u\bk{n}\neq1$, since $f\p{g\p{n}}=n$.  If $n'=n+1$, then $u\bk{n'}=u\bk{n+1}=1$, so $u\bk{n}=0\neq1$.  Either way, $u\bk{n}\neq1$ if $n$ has a preimage under $f$.  Similarly, $v\bk{n}\neq1$ if $n$ has a preimage under $f$.  Therefore, $u\bk{n}=1$ if and only if $v\bk{n}=1$ if and only if $n$ has no preimage under $f$.  So, $u$ and $v$ have their ones in the same positions.  Notice that this immediately implies that they also have their zeroes in the same positions, since in both $u$ and $v$, zeroes and ones appear in consecutive pairs.

Now, assume that $1<\ell<k+1$, and assume that $u$ and $v$ have all occurences of values less than $\ell$ in the same positions.  We will now show that all occurrences of $\ell$ in $u$ and $v$ are in the same positions.  Notice that in both $u$ and $v$, all occurrences of $\ell$ immediately follow occurrences of $\ell-1$.  But, it is also possible for the value after $\ell-1$ to be $k+1$.  We will show that if $u\bk{n-1}=v\bk{n-1}=\ell-1$, then $u\bk{n}=k+1$ if and only if $v\bk{n}=k+1$.  (This will suffice to show that the $\ell$'s coincide.)  We will build a chain of biconditionals.  We have
\begin{align}
\label{eq:if1}&\hspace{0.25in}u\bk{n}=k+1\\
\label{eq:if2}&\iff w\bk{g\p{n}}=k\\
\label{eq:if3}&\iff g\p{n}\text{ $2$-volatile with respect to $\seq{a_i}$}\\
\label{eq:if4}&\iff f\p{g\p{n}}\text{ $2$-volatile with respect to $\seq{b_i}$}\\
\label{eq:if5}&\iff n\text{ $2$-volatile with respect to $\seq{b_i}$}\\
\label{eq:if6}&\iff v\bk{n}=k+1.
\end{align}
We see that $\p{\ref{eq:if1}}\iff\p{\ref{eq:if2}}$ holds by the definition of $g$.  Then, $\p{\ref{eq:if2}}\iff\p{\ref{eq:if3}}$ holds by the definition of a representation word.  Next, $\p{\ref{eq:if3}}\iff\p{\ref{eq:if4}}$ holds by the definition of $f$ and of $1$-promotion.  Then, $\p{\ref{eq:if4}}\iff\p{\ref{eq:if5}}$ holds by item~\ref{it:tpfg} above.  And last, $\p{\ref{eq:if5}}\iff\p{\ref{eq:if6}}$ holds by the definition of a representation word.

Finally, notice that all corresponding positions in $u$ and $v$ are either equal and less than $k+1$ or unassigned in both.  These are both words over $\Sigma_{k+1}$, so all remaining positions in both $u$ and $v$ must have the value $k+1$.  Therefore, $u=v$, as required.
\end{proof}
Next, we will construct correspondences between Nim sequences that are representation words with different values of $k$.  This will show that ternary representation words that are Nim sequences ($k=2$) give rise to Nim sequences for all larger values of $k$.
We have the following theorem.
\begin{theorem}\label{thm:teruniv}
Let $\seq{a_i}$ be a representing sequence with $a_1=k$, and let $w$ be the representation word of $\seq{a_i}$.  Let $f$ be the $1$-promotion function of $\seq{a_i}$.  If $w$ is the Nim sequence of a subtraction set $S$ that contains $\bk{k-1}$, then $\varphi_k\p{w}$ is the Nim sequence of the subtraction set $S'=f\p{S}\cup\st{1}$.
\end{theorem}
\begin{proof}
Let $\seq{b_i}$ be the $1$-promotion of $\seq{a_i}$, and let $v=\varphi_k\p{w}$.  By Theorem~\ref{thm:tracepromo}, $v$ is the representation word of $\st{b_i}$.
As the Nim sequence of a subtraction game, $w$ is strongly Fergusonian, and $v$ is also strongly Fergusonian by construction.  Let $u$ denote the Nim sequence of $S'$.  We wish to show that $u=v$.

By Proposition~\ref{prop:scontt}, we can assume with loss of generality that $T\subseteq S$.  Let $T'=f\p{T}\cup\st{1}$.  We have that $T'\subseteq S'$.  By Theorem~\ref{thm:tracepromo}, $T'=\st{b_i-1\mid i\geq2}\cup\bk{k}$.  Let $N'=f\p{N}$, so that $N'$ is the set of non-volatile zends in the $\seq{b_i}$-representation.  By Theorem~\ref{thm:absmain}, to prove that $u=v$, it will suffice to show that $\p{N'+S'}\cap N'=\emptyset$, which is equivalent to showing that $\p{N'-S'}\cap N'=\emptyset$.


First, notice that $\p{N'-1}\cap N'=\emptyset$, since one less than a non-volatile zend is a volatile number (or a negative number).
We know that $v\bk{f\p{n}}=0$ if and only if $w\bk{n}=0$, and this happens if and only if $i\in N$.  We also know that $N'=f\p{N}$ and that $S'=f\p{S}\cup\st{1}$.  So, to prove that $\p{N'-S'}\cap N'=\emptyset$, it will suffice to show that for each $n\in N$, for all $s\in S$ with $s<n$, $v\bk{f\p{n}-f\p{s}}\neq 0$.
We will actually prove the stronger statement that for all $s\in S$, $w\bk{n-s}=v\bk{f\p{n}-f\p{s}}$.  Since $w$ is the Nim sequence of $S$, $w\bk{n-s}\neq 0$.  So, this statment is indeed stronger than the desired one.

We now note that all zends in $S$ are volatile (and hence $2$-volatile by Proposition~\ref{prop:zv2}), since any other zend $t$ would have $w\bk{t}=0$.  A legal move would be to subtract $t$, but $w\bk{0}=0$.  Now, fix $s\in S$.  We will prove that $w\bk{n-s}=v\bk{f\p{n}-f\p{s}}$ for all $n\in N$ with $n\geq s$ by demonstrating a combinatorial property of the function $f$.  Let $z$ be the number of zeroes in $w$ before position $s$.  Let $n_m$ be the $\p{m+1}^{st}$ zero in $w$ after position $s$.  We will show that for all $m\geq0$, $f\p{n_m}-f\p{s}=f\p{n_m-s}+1$.  Since $w\bk{n_m-s}\neq0$, we would then have that $v\bk{f\p{n_m-s}}=w\bk{n_m-s}+1$.  So, we would have $v\bk{f\p{n_m}-f\p{s}}=w\bk{n_m-s}$, as required.  Notice that $f\p{s}=s+z$ and $f\p{n_m}=n_m+z+m$ as a consequence of Theorem~\ref{thm:tracepromo}.  The proof of our claim will be by induction on $m$ for fixed $s$.

Continue to fix $s\in S$, and let $m=0$.  Since $0<n_0-s<k$, we have that $f\p{n_0-s}=n_0-s+1$.  So, we see that
\[
f\p{n_0}-f\p{s}=n_0+z-s-z=n_0-s=f\p{n_0-s}-1.
\]
Hence, our claim holds for $m=0$.

Now, assume that $f\p{n_{m-1}}-f\p{s}=f\p{n_{m-1}-s}+1$.  We have $n_m=i_{m-1}+k+c$, where $c\in\st{0,1}$.  We then have that $n_m-s=n_{m-1}-s+k+c$.  Notice that, since $w\bk{n_m-s}\neq0$, there must be exactly $m+1$ zeroes preceding position $n_m-s$ in $w$.  Hence, $f\p{n_m-s}=n_m-s+m+1=n_{m-1}-s+k+c$.  So, we see that
\[
f\p{n_m}-f\p{s}=n_m+z+m-s-z=n_m-s+m=f\p{n_m-s}-1.
\]
Hence, our claim holds for $m$, as required.
\end{proof}
It is unclear how strong of a converse Theorem~\ref{thm:teruniv} has.  We know that, while applying an inverse map to $\varphi_k$, one does obtain a representation word, that representation word need not be a Nim sequence.  For example, the representing sequence $\p{1,3,10,14,18}$ would get \quot{un-promoted} to $\p{1,2,7,10,13}$.  The representation word of the former is (the period of) a Nim sequence; the representation word of the latter is not.  It is open to characterize when we can reduce the alphabet size while preserving the fact that we have a Nim sequence.

All of the work of this appendix has been leading up to the following extension of Theorem~\ref{thm:main}.
\begin{theorem}\label{thm:mainext}
Let $k\geq2$, and let $\seq{a_{i,k}}$ be the representing sequence where $a_{0,k}=1$, $a_{1,k}=k$, and $a_{i,k}=3a_{i-1,k}-a_{i-2,k}$ for all $i\geq2$.  Let $w_k$ be the representation word of $\seq{a_{i,k}}$.  There exists a subtraction set $S$ such that $w_k$ is the Nim sequence of $S$.  In particular, there exist aperiodic, bounded Nim sequences with maximum Sprague-Grundy value equal to $k$ for all $k\geq2$.
\end{theorem}
\begin{proof}
The proof will be by induction on $k$.  The case $k=2$ follows from Theorem~\ref{thm:main}.  Now, assume the statement is true for $k-1$.  Furthermore, assume that the $\seq{a_{i,k-1}}$-representation of $a_{i,k-1}-1$ is $21^{i-1}0$ for all $i\geq1$ (the base case being covered by Lemma~\ref{lem:oddfiba}).  To show the truth of the main statement for $k$, Theorem~\ref{thm:teruniv} implies that it will suffice to show that the $1$-promotion of $\st{a_{i,k-1}}_{i\geq0}$ is $\st{a_{i,k}}_{i\geq0}$.  Let $\seq{b_i}$ be the $1$-promotion of $\st{a_{i,k-1}}_{i\geq0}$. We will prove that $a_{i,k}=b_i$ by induction on $i$.  If $i=0$, then we have $b_0=a_{0,k}=1$, as required.  If $i=1$, then we have $b_1=1+\p{\p{k-1}-1+1}\cdot1=k$, as required.  Now, assume that $b_{i-1}=a_{i-1,k}$ for some $i\geq2$.  Since the $\seq{a_{i,k-1}}$-representation of $a_{i,k-1}-1$ is $21^{i-1}0$, we have
\[
b_i=1+\sum_{\ell=1}^{i-2}b_\ell+2b_{i-1}.
\]
Notice that this is reminiscent of recurrence~\ref{rec:5} in Lemma~\ref{lem:fibrec}.  Recurrence~\ref{rec:2} can be derived from recurrence~\ref{rec:5} regardless of initial conditions, so it follows that $b_i=3b_{i-1}-b_{i-2}$.  In other words, $b_i=a_{i,k}$, as required.

The fact that the $\seq{a_{i,k}}$-representation of $\seq{a_{i,k}-1}$ is $21^{i-1}0$ for all $i\geq1$ follows from subtracting $1$ from the fact that
\[
a_{i,k}=1+\sum_{\ell=1}^{i-2}a_{\ell,k}+2a_{i-1,k}.
\]
\end{proof}

\end{document}